\documentclass{amsart}
\usepackage{amsmath,amssymb,mathrsfs,yhmath,longtable,mathtools,hyperref}
\usepackage{amsthm,rotating,color,epsfig,mdframed,url,bbm}

\usepackage{graphicx}

\date{}

\theoremstyle{theorem}
\newtheorem{theo}{Theorem}
\newtheorem{lemm}[theo]{Lemma}
\newtheorem{prop}[theo]{Proposition}
\newtheorem{coro}[theo]{Corollary}

\theoremstyle{definition}
\newtheorem{defi}[theo]{Definition}

\theoremstyle{remark}
\newtheorem{rema}{Remark}

\graphicspath{{./pictures/}}	

\author {Maxim Prasolov and Vladimir Shastin}
\title[Distinguishing Legendrian knots of topological type~$7_4$, $9_{48}$ and $10_{136}$]{Distinguishing Legendrian knots in topological types~$7_4$, $9_{48}$, $10_{136}$ with maximal Thurston--Bennequin number}
\address{\noindent Department of Mechanics and Mathematics of Moscow State University, 1 Leninskije gory, Moscow 119991, Russia}
\email{0x00002a@gmail.com}
\address{Department of Mechanics and Mathematics of Moscow State University, 1 Leninskije gory, Moscow 119991, Russia}
\email{vashast@gmail.com}

\begin{document}
\maketitle

\begin{abstract}
In a recent work of I.\,Dynnikov and M.\,Prasolov\cite{dypra202?} a new method of comparing Legendrian knots
with nontrivial symmetry group is proposed. Using this method we confirm conjectures of~\cite{chong2013} about Legendrian knots in topological types~$7_4$, $9_{48}$ and $10_{136}$. This completes the classification of Legendrian types of rectangular diagrams of knots of complexity up to 9.
\end{abstract}

\tableofcontents

\section{Introduction}

We demonstrate the method introduced in \cite{dypra202?} to distinguish Legendrian knots. 

We denote by $\xi_+$ the standard contact structure on $\mathbb{S}^3$ and by $\xi_-$ the mirror image of $\xi_+$. 

Let $R$ be an oriented rectangular diagram of a link with a numbering of connected components. This diagram determines a cusp-free piecewise smooth link in $\mathbb{S}^3$ whose components are oriented and numbered. We denote this link by $\widehat R$. This link is $\xi_+$-Legendrian and $\xi_-$-Legendrian simultaneously. We denote by $\mathscr L_+ (R)$ (respectively, $\mathscr L_- (R)$) the class of $\xi_+$-Legendrian (respectively, $\xi_-$-Legendrian) equivalence of the link $\widehat R$.
We have $\mathscr L_+(R_1) = \mathscr L_+(R_2)$ (respectively, $\mathscr L_-(R_1) = \mathscr L_-(R_2)$) if and only if (see~\cite{ngth,OST}) $R_1$ and $R_2$ are related by a sequence of elementary moves of two kinds:
\begin{enumerate}
\item exchange moves;
\item stabilizations and destabilizations of type~I (respectively, of type~II).
\end{enumerate}

The set of all classes of combinatorial equivalence of rectangular diagrams obtained from $R$ by exchange moves is called {\it exchange class} and is denoted by $[R]$. If $[R_1] = [R_2]$ then $\mathscr L_{\pm} (R_1) = \mathscr L_{\pm} (R_2).$ Let us discuss the converse statement.

Let $L_1$ and $L_2$ be links in $\mathbb{S}^3$. A {\it morphism} from $L_1$ to $L_2$ is a connected component of the space of homeomorphisms $(\mathbb S^3, L_1)\to(\mathbb S^3,L_2)$ which preserve the orientation of the sphere, and the orientations and the numbering of all connected components of the links. Morphisms from $L$ to $L$ form a group under composition which we call the {\it symmetry group} and denote by $\mathrm{Sym}(L).$ We write $\mathrm{Sym} (\widehat R)$ simply by $\mathrm{Sym} (R).$

With every exchange move, stabilization or destabilization that produces a rectangular diagram $R_2$ from $R_1$ we associate a morphism from $\widehat R_1$ to $\widehat R_2$. A subset of $\mathrm{Sym} (R)$ consisting of elements which can be represented by a composition of morphisms associated with exchange moves, stabilizations or destabilizations of type~I (respectively, type~II) is a subgroup and is denoted by $\mathrm{Sym}_+ (R)$ (respectively, $\mathrm{Sym}_- (R)$).

\begin{theo}[Corollary 4.9 of \cite{dypra202?}]
\label{exchange-count}
For any rectangular diagram $R$ the number of exchange classes $[R']$ such that $\mathscr L_{\pm} (R) = \mathscr L_{\pm} (R')$ is equal to the cardinality of the double coset $\mathrm{Sym}_- (R) \backslash \mathrm{Sym}(R) / \mathrm{Sym}_+(R)$.
\end{theo}

In this paper we deal with the situation when the double coset consists of one element. So $\mathscr L_- (R) = \mathscr L_- (R')$ and $[R]\neq [R']$ implies that $\mathscr L_+ (R) \neq \mathscr L_+ (R').$ We use this to distinguish Legendrian types $\mathscr L_+ (R)$ and $\mathscr L_+ (R')$. The condition $\mathscr L_- (R) = \mathscr L_- (R')$ always can be ensured by appropriate choice of rectangular diagrams $R$ and $R'$ preserving their $\xi_+$-Legendrian types (see \cite{bypasses}). The condition $[R]\neq [R']$ can be checked by generating a finite list of all combinatorial types of rectangular diagrams obtained from $R$ by exchange moves.

By the discussion in \cite[Section "Examples"]{dypra202?}, the last unresolved conjectures of~\cite{chong2013} are about topological types $7_4,$ $9_{48}$ and $10_{136}.$ We apply this method to these cases. We emphasize that the main result of the present paper is contingent on the results of the paper \cite{dypra202?} which is not published yet.

\begin{prop}
\label{main-result-7_4}
$\mathscr L_+ (7_4^1) \neq \mathscr L_+ (7_4^2),$
\end{prop}

\begin{figure}[h]
\includegraphics[scale=.2]{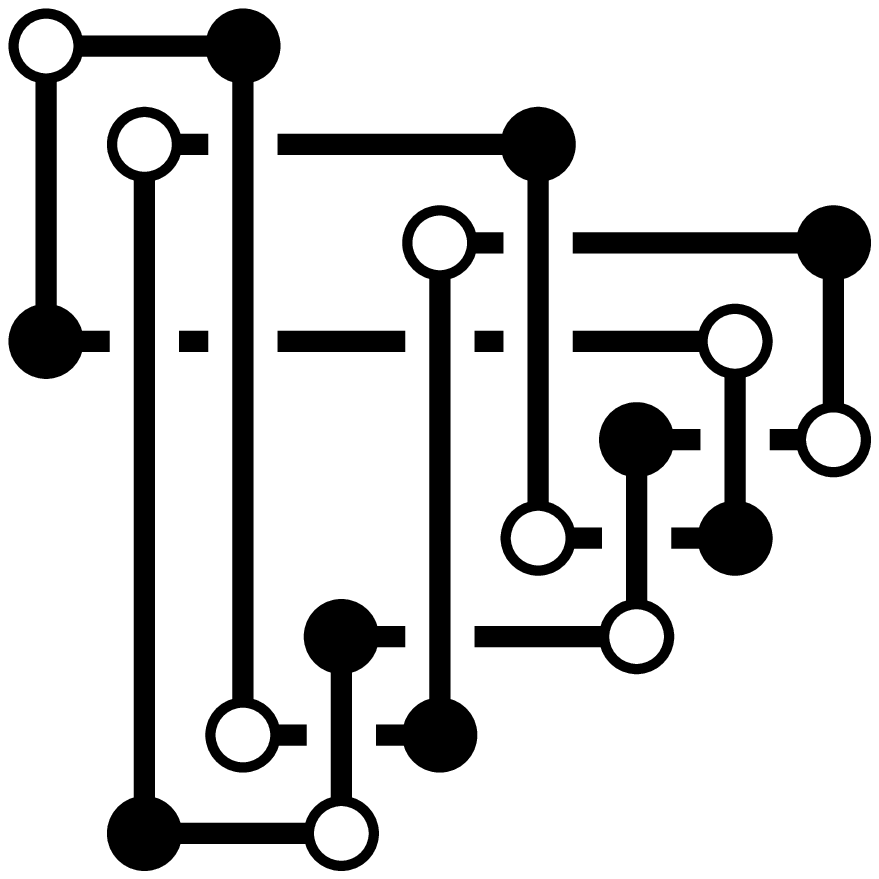}\hskip1cm
\includegraphics[scale=.2]{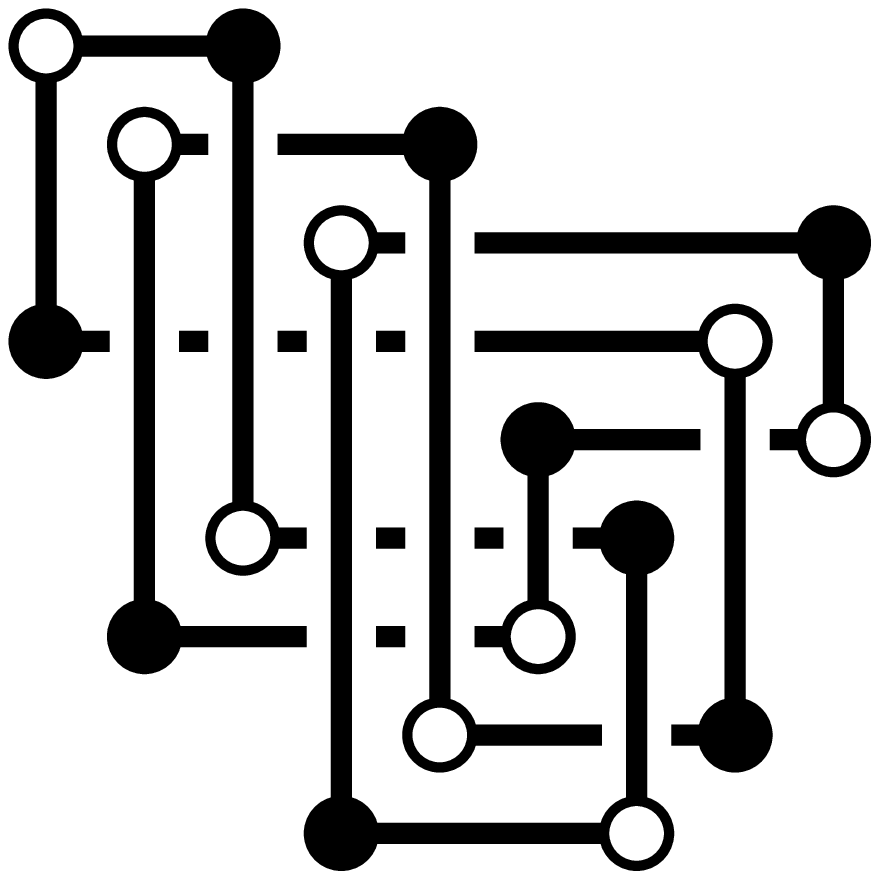}
\put(-105,-20){$7_4^1$}
\put(-30,-20){$7_4^2$}
\end{figure}

\begin{prop}
\label{main-result-9_48}
$\mathscr L_+ (9_{48}^1) \neq \mathscr L_+ (9_{48}^2)$ and $\mathscr L_+ (9_{48}^3) \neq \mathscr L_+ (9_{48}^4)$,
\end{prop}

\begin{figure}[h]
\includegraphics[scale=.2]{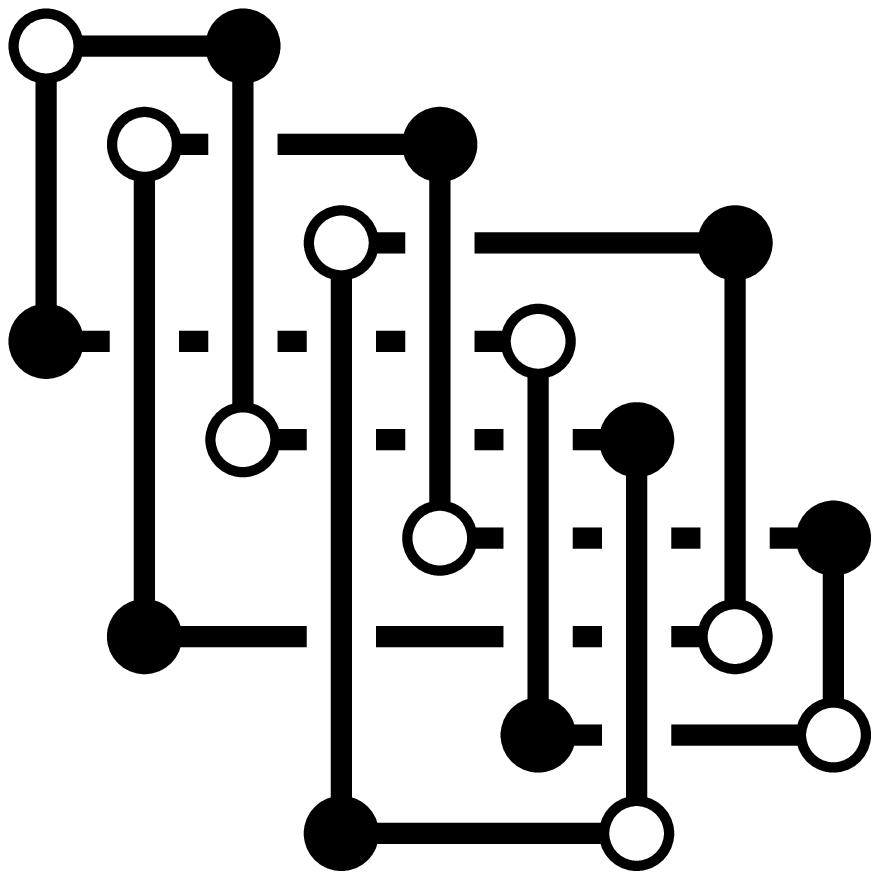}\hskip1cm
\includegraphics[scale=.2]{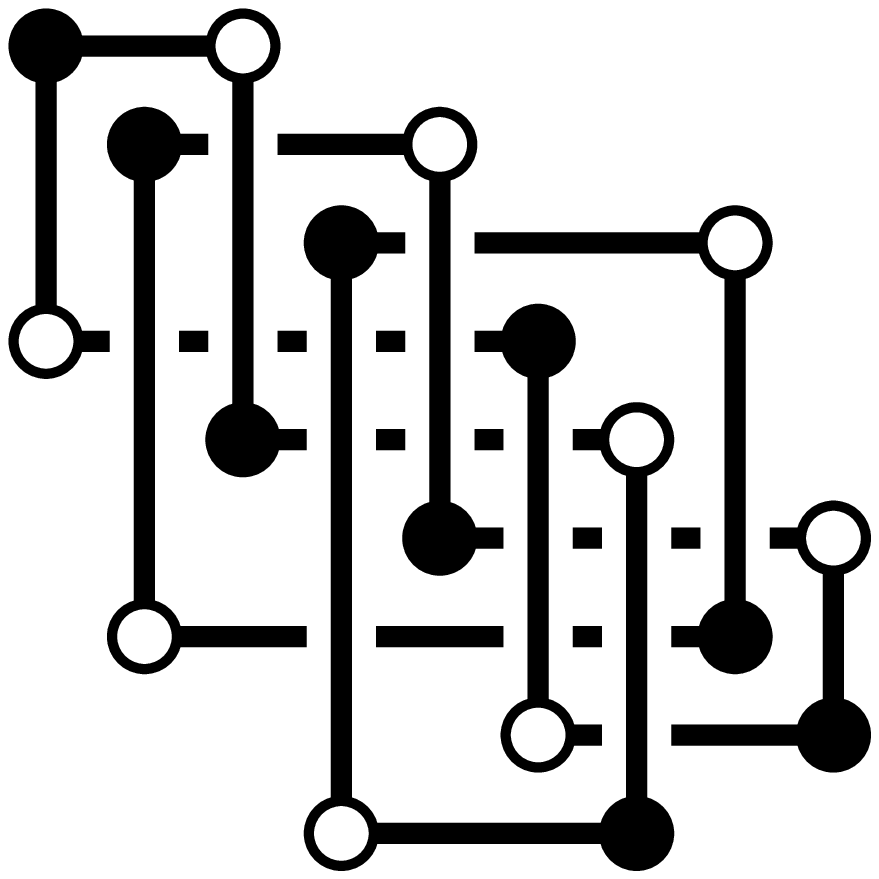}\hskip3cm
\includegraphics[scale=.2]{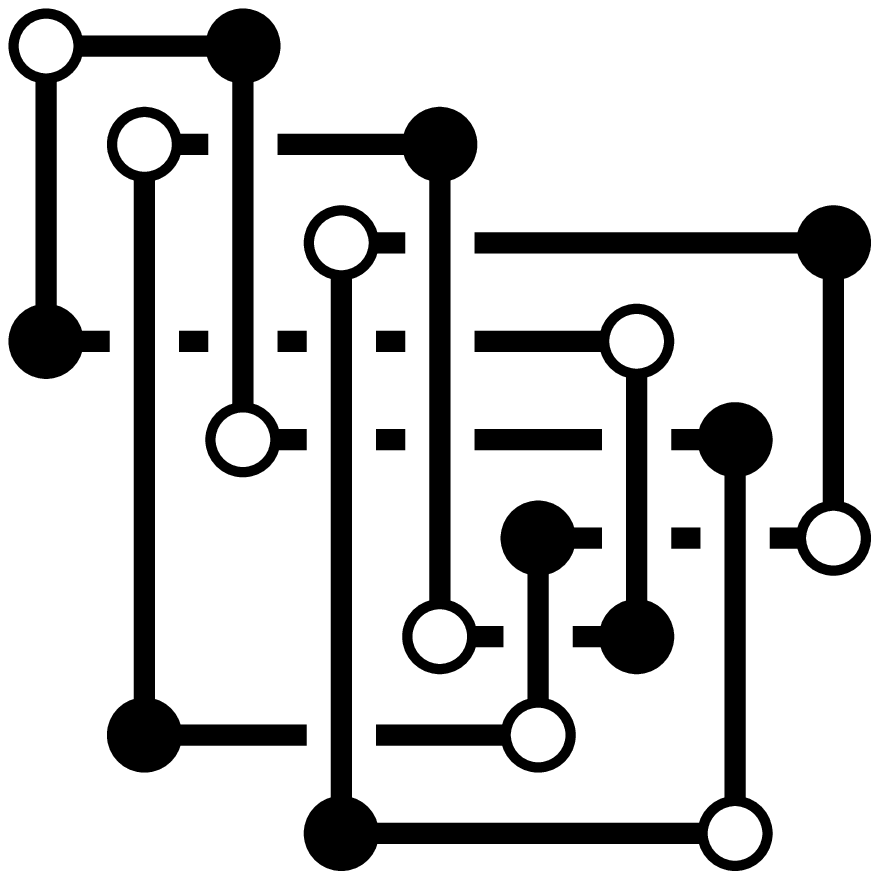}\hskip1cm
\includegraphics[scale=.2]{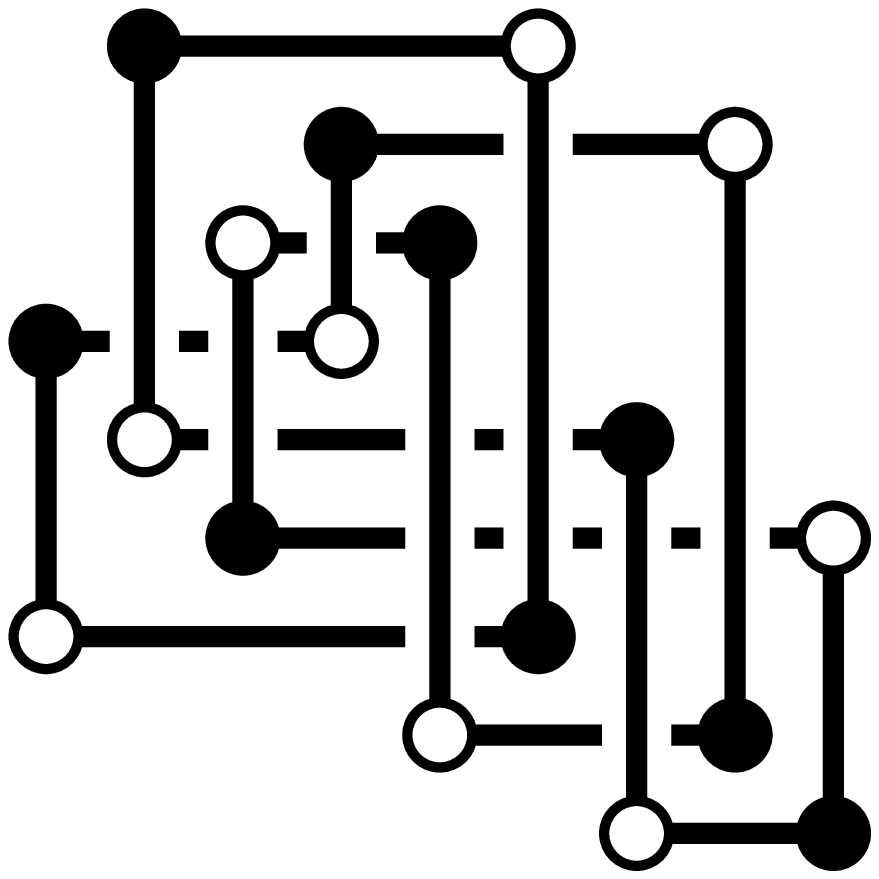}
\put(-320,-20){$9_{48}^1$}
\put(-245,-20){$9_{48}^2$}
\put(-105,-20){$9_{48}^3$}
\put(-35,-20){$9_{48}^4$}
\end{figure}

\begin{prop}
\label{main-result-10_136}
$\mathscr L_+ (10_{136}^1) \neq \mathscr L_+ (10_{136}^2)$ and $\mathscr L_+ (10_{136}^3) \neq \mathscr L_+ (10_{136}^4)$.
\end{prop}

\begin{figure}[h]
\includegraphics[scale=.2]{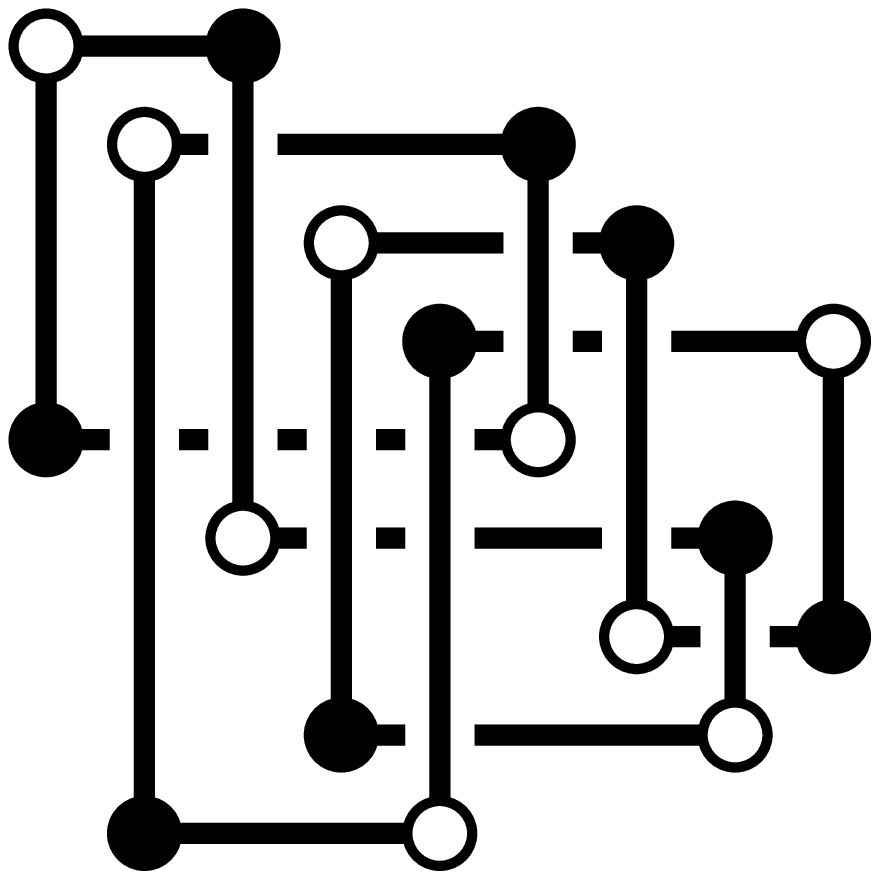}\hskip1cm
\includegraphics[scale=.2]{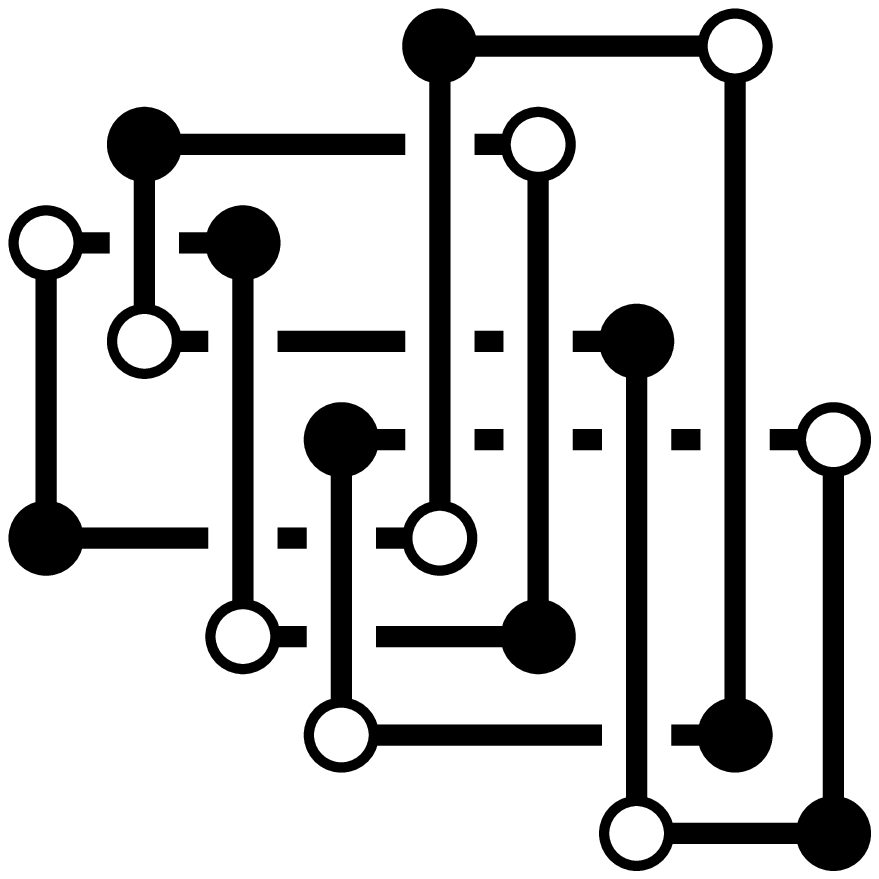}\hskip3cm
\includegraphics[scale=.2]{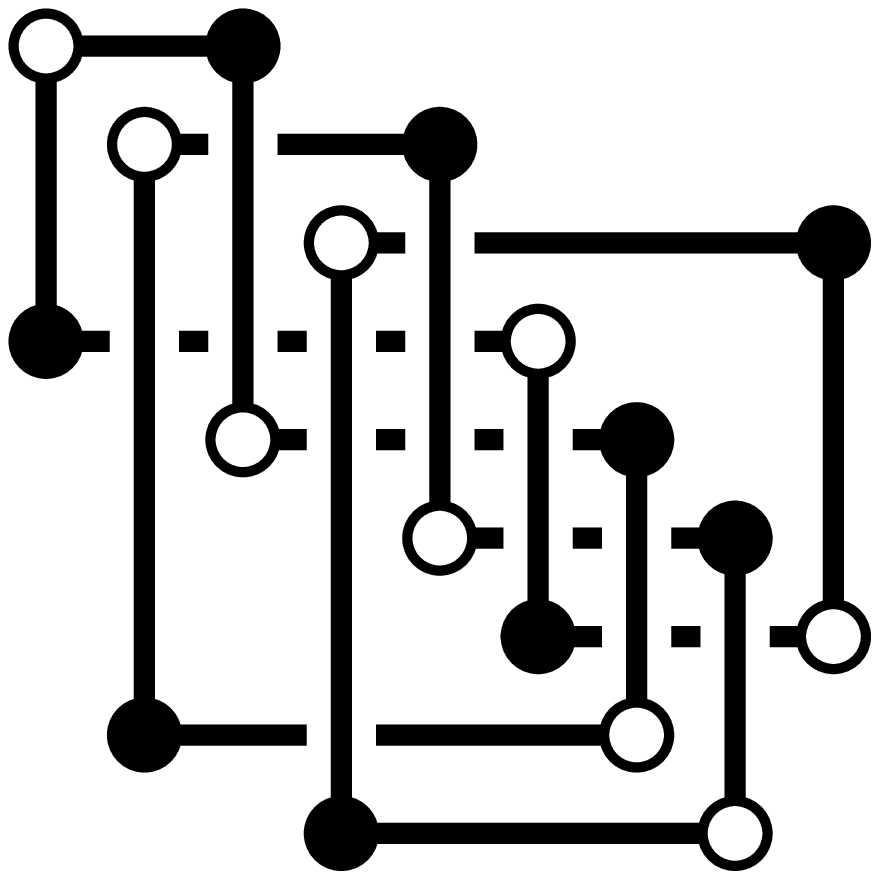}\hskip1cm
\includegraphics[scale=.2]{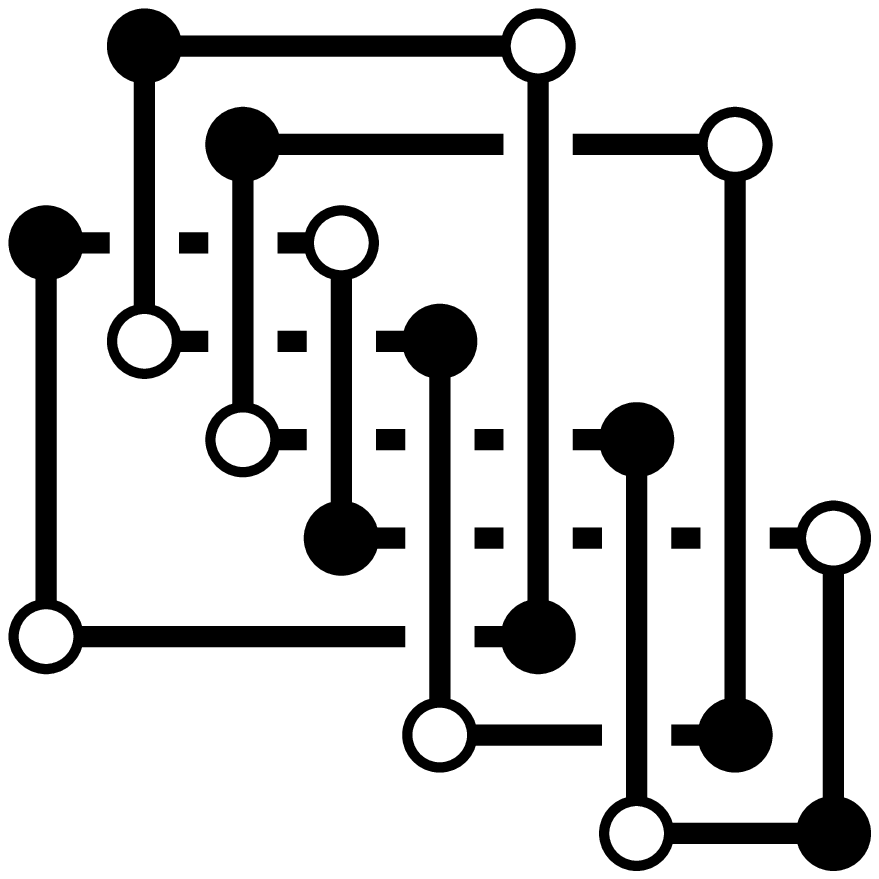}
\put(-320,-20){$10_{136}^1$}
\put(-245,-20){$10_{136}^2$}
\put(-105,-20){$10_{136}^3$}
\put(-35,-20){$10_{136}^4$}
\end{figure}

In section~\ref{realizing-symmetry} we show that $\mathrm{Sym}(R) = \mathrm{Sym}_-(R)$ for some rectangular diagram $R$ in each topological type $7_4,$ $9_{48}$ and $10_{136}$. In all these cases the symmetry group is cyclic and we represent a generator of each group by a composition of morphisms associated with exchange moves, stabilizations and destabilizations of type~II.  

Using this, in section~\ref{distinguishing} we show that $\mathrm{Sym}(R) = \mathrm{Sym}_-(R)$ for any diagram from the list $7_4^1, 9_{48}^1, 9_{48}^3, 10_{136}^1$ or $10_{136}^3$. Then we argue in the same way as in \cite{dysha202?} where the case of the trivial orientation-preserving symmetry group is considered.

\section{Some notations, definitions and conventions}
The reader is referred to \cite{dypra202?} for the terminology that we use here.

Let us recall that every elementary move is associated with a rectangle $r$ on the two-dimensional torus $\mathbb T^2$. Denote the set of all vertices of $r$ by $\partial r.$ The intersection of the rectangle with the set of the vertices of the rectangular diagram can only be a subset of $\partial r$. In the case of {\it exchange move} this intersection is a pair of two adjacent vertices of the rectangle, in the case of {\it stabilization} --- a single vertex, in the case of {\it destabilization} --- three vertices. We distinguish two types of (de)stabilizations, see Figure~\ref{elementary-moves}.

\begin{figure}[h]
\includegraphics[scale=.2]{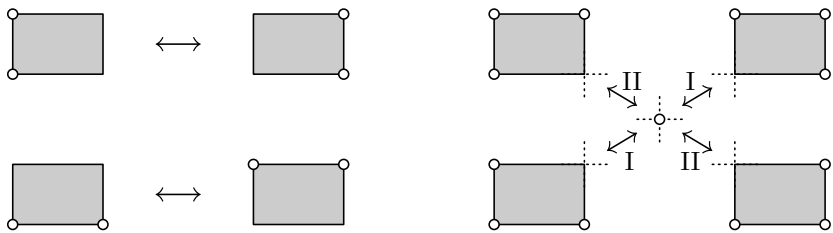}
\put(-55,53){I}
\put(-79,53){II}
\put(-78,23){I}
\put(-57,23){II}
\caption{Exchange moves and (de)stabilizations of type~I and~II}
\label{elementary-moves}
\end{figure}

To apply the move we delete from the rectangular diagram common vertices with the rectangle and add the other vertices of the rectangle. The obtained rectangular diagram inherits the orientation and numbering of the components.

We write $R_1\stackrel{s}{\mapsto}R_2$ for an elementary move $s$ from the diagram $R_1$ to $R_2$.

We associate a morphism with every elementary move using the following definition (see~\cite{dypra202?}).

\begin{defi}
Let~$L_1$ and~$L_2$ be two links in~$\mathbb S^3$, and let~$\eta$ be a morphism from~$L_1$
to~$L_2$. Suppose that there is an embedded two-disc~$d\subset\mathbb S^3$ such that
the following holds:
\begin{enumerate}
\item
the symmetric difference~$L_1\triangle L_2$ is a union of two open arcs~$\alpha\subset L_1$,
$\beta\subset L_2$;
\item
$\partial d=\overline{\alpha\cup\beta}$;
\item
the morphism~$\eta$ is represented by a homeomorphism~$(\mathbb S^3,L_1)\rightarrow(\mathbb S^3,
L_2)$ that is identical outside an open three-ball~$B$ containing the interior of~$d$
and intersecting~$L_1\cup L_2$ in~$\alpha\cup\beta$ (see Figure~\ref{alpha-beta-fig}).
\end{enumerate}
Then we say that the triple~$(L_1,L_2,\eta)$ is \emph{a $\mathbb D^2$-move associated with~$d$}.
\end{defi}
\begin{figure}[ht]
\centerline{\includegraphics{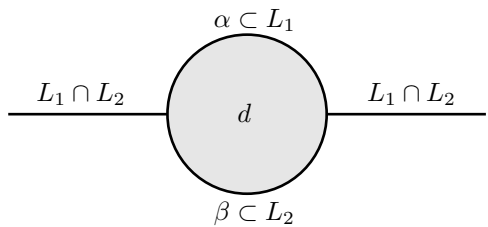}\put(-180,40){$L_1\cap L_2$}
\put(-55,40){$L_1\cap L_2$}\put(-104,32){$d$}\put(-113,68){$\alpha\subset L_1$}
\put(-113,-5){$\beta\subset L_2$}}
\caption{A $\mathbb D^2$-move}\label{alpha-beta-fig}
\end{figure}

We use the coordinate system $(\theta, \varphi, \tau)$, $\theta,\varphi\in\mathbb R/2\pi\mathbb Z$ and $\tau\in[0;1]$ on $\mathbb S^3$, where $(\theta_1,\varphi,0)=(\theta_2,\varphi,0)$ and $(\theta, \varphi_1, 1)=(\theta, \varphi_2, 1)$, coming from the join construction $\mathbb S^3 \cong \mathbb S^1 * \mathbb S^1.$ For a point $v(\theta_0, \varphi_0)\in\mathbb T^2$ we denote by $\widehat v$ the segment $\{\theta_0\}\times\{\varphi_0\}\times[0;1].$ For a finite subset $V\subset\mathbb T^2$ we denote by $\widehat V$ the union $\bigcup\limits_{v\in V} \widehat v.$ With a rectangle $r=[\theta_0;\theta_1]\times[\varphi_0;\varphi_1]$ we associate the tetrahedron $[\theta_0;\theta_1]*[\varphi_0;\varphi_1]$ and denote it by $\Delta_r$.

If an elementary move $R_1 \stackrel{s}{\mapsto} R_2$ is associated with the rectangle $r$ then with this move we associate a morphism $\widehat s$ from $\widehat{R_1}$ to $\widehat{R_2}$ such that the triple $(\widehat{R_1},\widehat{R_2},\widehat s)$ is a $\mathbb D^2$-move associated with some disk contained in the tetrahedron $\Delta_r$. Clearly the boundary of this disk is $\widehat {\partial r}.$

We orient connected components of a link $\widehat{R}$ by assigning to each vertex of the diagram $R$ a color: a vertex $v$ is black (white) if the edge $\widehat v$ is oriented in the direction of increasing (decreasing) of $\tau$. 

We denote by $r_{\scriptscriptstyle\diagdown}: \mathbb T^2 \to \mathbb T^2$ the reflection about the line $\theta+\varphi=0$: $r_{\scriptscriptstyle\diagdown}(\theta, \varphi) = (-\varphi, -\theta),$ and by $\widehat r_{\scriptscriptstyle\diagdown}:\mathbb S^3 \to \mathbb S^3$ --- the homeomorphism $\widehat r_{\scriptscriptstyle\diagdown}(\theta, \varphi, \tau) = (-\varphi, -\theta, 1-\tau).$

For a homeomorphism $h:\mathbb S^3\to\mathbb S^3$ and a link $L$ we denote by $[h]_L$ the morphism from $L$ to $h(L)$ which is the connected component containing $h$. We say that the homeomorphism $h$ {\it represents} the morphism $[h]_L.$

By a {\it combinatorial equivalence} between two rectangular diagrams $R_1$ and $R_2$ we mean a pair of orientation-preserving self-homeomorphisms of $\mathbb S^1,$ which acts on the torus $\mathbb T^2$ in the natural way, mapping $R_1$ to $R_2$ and preserving colors of vertices and numbering of connected components. The class of combinatorial equivalence of a diagram $R$ we denote also by $R$.

Every pair $(f,g)$ of orientation-preserving self-homeomorphisms of $\mathbb S^1$ induces a homeomorphism $\widehat{(f,g)}:\mathbb S^3\to\mathbb S^3$ by the rule: $(\theta,\varphi,\tau)\mapsto(f(\theta),g(\varphi),\tau).$ 

Unless otherwise specified, on pictures of rectangular diagrams we draw a portion of the torus $\mathbb T^2$ not containing any point with $\theta=0$ or $\varphi=0.$

\section{Realizing the symmetry group by type~II elementary moves}
\label{realizing-symmetry}

\begin{figure}[h]
\includegraphics[scale=.2]{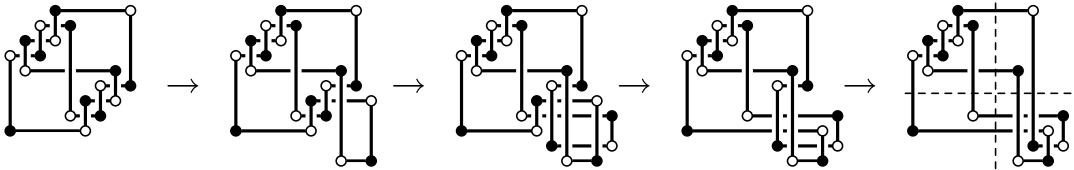}
\caption{A representation of a generator of the symmetry group in the topological type $7_4$ by a composition of two stabilizations and two destabilizations of type~II}
\label{7_4-flype}
\end{figure}

\begin{lemm}
\label{7_4-symmetry}
$\mathrm{Sym}_-(7_4^1) = \mathrm{Sym}(7_4^1).$
\end{lemm}

\begin{proof}
The diagram $7_4^1$ is combinatorially equivalent to both diagrams on the left and on the right of Figure~\ref{7_4-flype}. There exists a combinatorial equivalence between the mentioned diagrams in this figure which maps the dashed lines to the lines $\theta=0$ and $\varphi=0.$ A composition of all morphisms in Figure~\ref{7_4-flype} and a morphism induced by the combinatorial equivalence we denote by~$\eta$. By Lemma~\ref{combinatorial-morphism}, $\eta\in\mathrm{Sym}_-(7_4^1).$

The knot $7_4$ is two-bridge of type $(15,11)$. By Theorem~4.1 of \cite{sak90} its symmetry group is isomorphic to $\mathbb Z_4.$ Let us prove that $\eta$ is a generator of this group. It is sufficient to prove that $\eta^2$ is not trivial.


\begin{figure}[h]
\includegraphics[scale=.3]{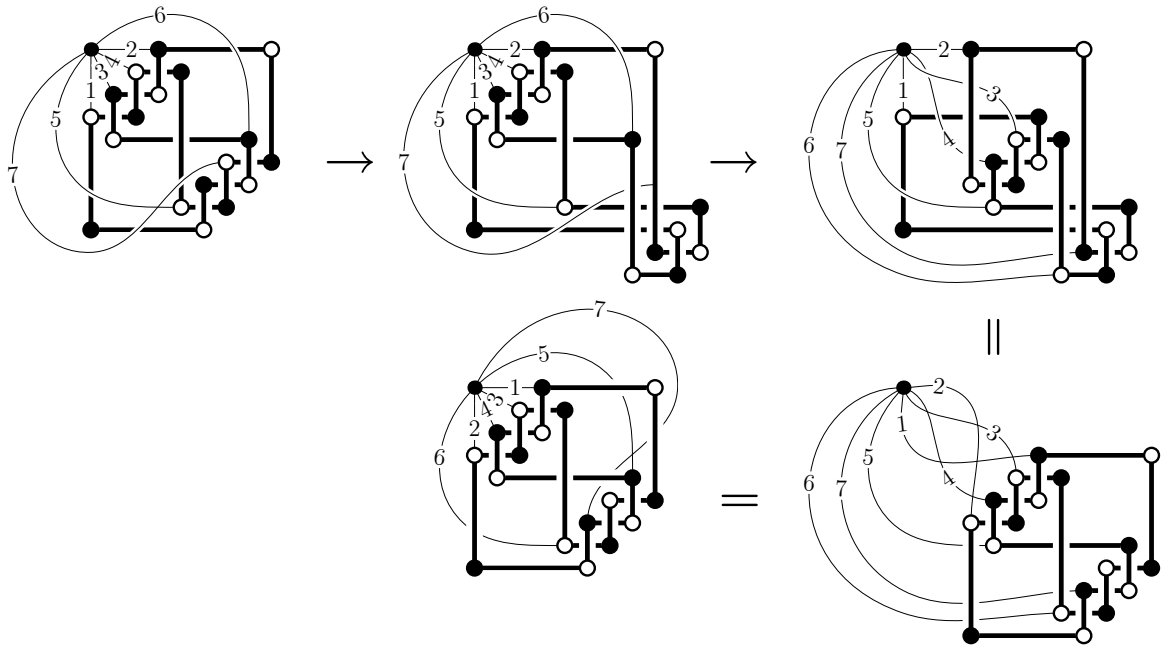}
\caption{The action on the fundamental group of a homeomorphism representing $\eta^2$}
\label{loops7_4}
\end{figure}

We represent the standard Wirtinger generators of the group $G(7_4) = \pi_1(\mathbb{S}^3 \setminus \widehat{7_4^1} , p)$ by loops on the left of Figure~\ref{loops7_4}. The coordinates of the base point $p$ are $(\pi, \pi, \dfrac{1}{2})$.

We illustrate each loop by an arc connecting the base point $p$ with some point on the knot. The corresponding loop consists of three parts. The second part is a small loop in the neighborhood of the point on the knot which has a linking number $+1$ with the knot. Both the first part and the third part of the loop are the portion of the arc connecting the base point with the small loop. 


The action of a homeomorphism representing the $\eta^2$ on the loops is shown  in Figure~\ref{loops7_4}. The "=" means a homeomorphism isotopic (relative to the link) to a homeomorphism induced by the combinatorial equivalence. The "$\to$" means a homeomorphism representing $\eta$ but without an application of the combinatorial equivalence which we postpone. Before and after application of a homeomorphism we are free to homotope the loops to avoid their intersection with disks of $\mathbb D^2$-moves. For each elementary move we assume that the ball $B$ in the definition of a $\mathbb D^2$-move is sufficiently small so that each loop stay fixed under the corresponding homeomorphism. 
%

It is clear from this figure that the action on the fundamental group of some homeomorphism representing $\eta^2$ coincides with the action of $\widehat r_{\scriptscriptstyle\diagdown}.$ By \cite[Corollary 7.5]{wald68} this homeomorphism is isotopic to $\widehat r_{\scriptscriptstyle\diagdown}.$ Since $\widehat r_{\scriptscriptstyle\diagdown}$ is of finite order, by Lemma~\ref{trivial-symmetry} the morphism $[\widehat r_{\scriptscriptstyle\diagdown}]_{\widehat{7_4^1}}$ is nontrivial because only torus knots have nontrivial center (see~\cite{buzie}) and the knot $7_4$ is not a torus knot.

\begin{rema}
In fact, it is unnecessary here to use the result of Waldhausen. Suppose that a perspective projection of a link to some plane is regular. The projection is a union of open arcs separated by undercrossings and  of simple closed curves. Connect the center of the projection with each arc and with each simple closed curve by a straight line segment. This union of the link with the segments we call {\it Wirtinger spatial graph of a link}. It is known that the complement to a Wirtinger spatial graph is a handlebody (this fact is especially familiar when the link is the Hopf link). In Figure~\ref{loops7_4} we prove that $\eta^2$ can be represented by a homeomorphism which coincides with the symmetry $\widehat r_{\scriptscriptstyle\diagdown}$ on the Wirtinger spatial graph. Since any homeomorphism of a handlebody which is identical on the boundary is isotopic to the identity, $\eta^2$ is represented by the symmetry.
\end{rema}

\end{proof}


\begin{figure}[h]
\includegraphics[scale=.2]{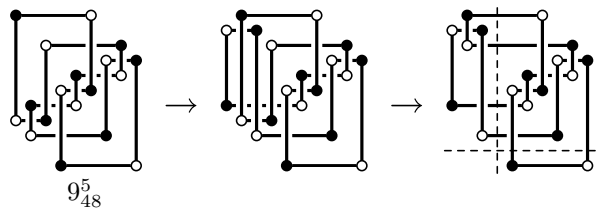}
\put(-200,-10){$9_{48}^5$}
\caption{A representation of a generator of the symmetry group in the topological type $9_{48}$ by a composition of a stabilization and a destabilization of type~II}
\label{9_48-flype}
\end{figure}

\begin{lemm}
\label{9_48-symmetry}
Let $9_{48}^5$ be the diagram on the left of Figure~\ref{9_48-flype}. Then $\mathrm{Sym}_-(9_{48}^5) = \mathrm{Sym}(9_{48}^5).$
\end{lemm}

\begin{proof}
There exists a combinatorial equivalence between the diagram on the right of Figure~\ref{9_48-flype} and the diagram on the left of the same figure which maps the dashed lines to the lines $\theta=0$ and $\varphi=0.$ A composition of all morphisms in Figure~\ref{9_48-flype} and a morphism induced by the combinatorial equivalence we denote by~$\eta$. By Lemma~\ref{combinatorial-morphism}, $\eta\in\mathrm{Sym}_-(9_{48}^5).$

The knot $9_{48}$ is nonelliptic Montesinos of type $(-1/3,2/3,2/3)$. By \cite[Theorem~1.3]{boizi} its symmetry group is isomorphic to $\mathbb Z_6.$ Let us prove that $\eta$ is a generator of this group. It is sufficient to prove that $\eta^2$ and $\eta^3$ are not trivial. We represent the morphism $\eta$ by a homeomorphism $T$ which preserves the base point $p$. We consider the action of $T$ on the fundamental group $G(9_{48})$ of the knot complement and prove that $T^2_*$ and $T^3_*$ are not inner homomorphisms.

We use conventions in the proof of Lemma~\ref{7_4-symmetry}.

In Figure~\ref{9_48loops} we define the homeomorphism $T$ to be a composition of three homeomorphisms fixing the base point. The first is the composition of homeomorphisms representing morphisms associated with the stabilization and destabilization. The second (the third) is isotopic relative to the link to the homeomorphism induced by a combinatorial equivalence which keeps horizontal (vertical) levels fixed.

\begin{figure}[h]
\begin{tabular}{lcr}
\includegraphics[scale=0.25]{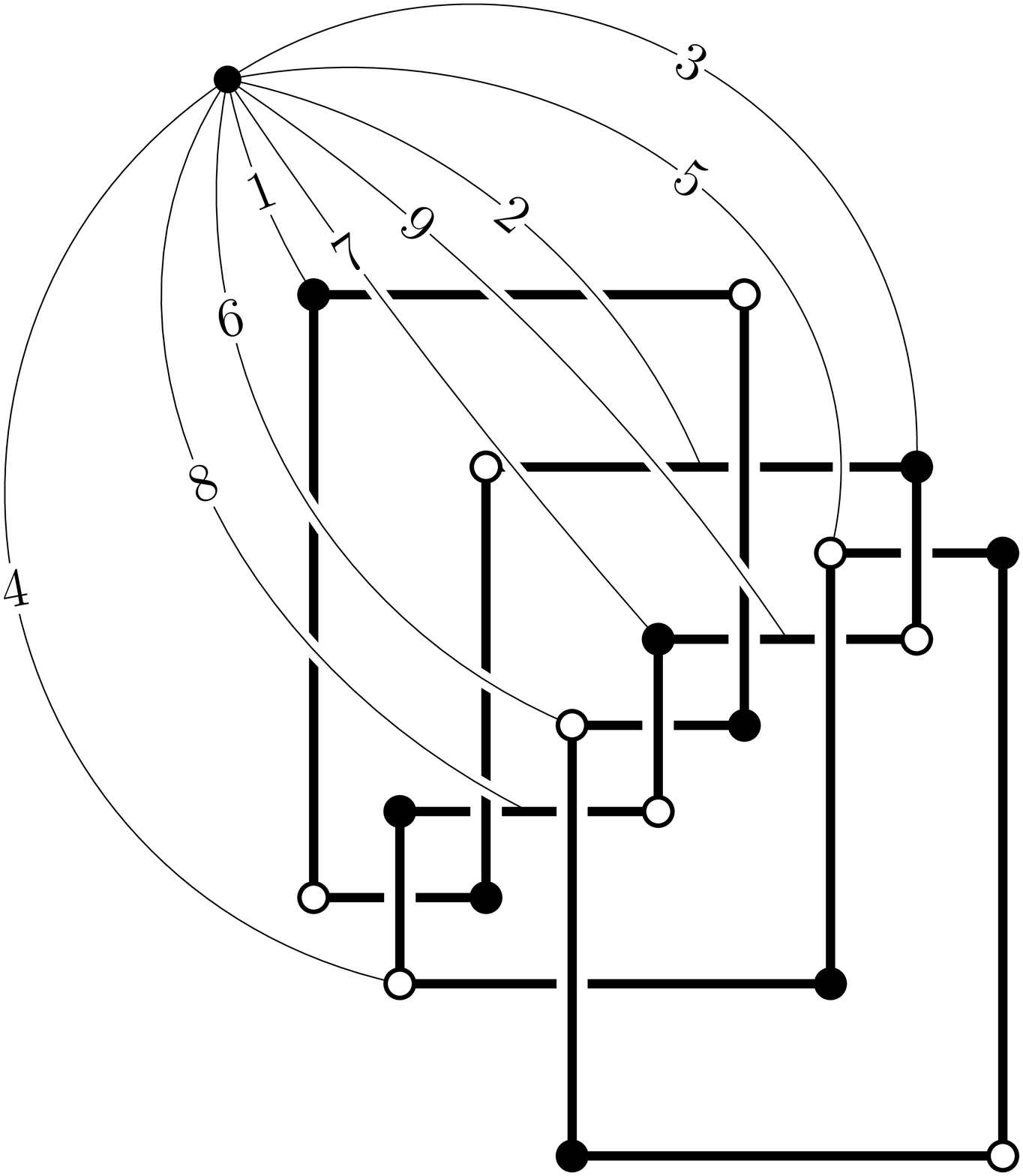}&$\quad\quad\quad$&\includegraphics[scale=0.25]{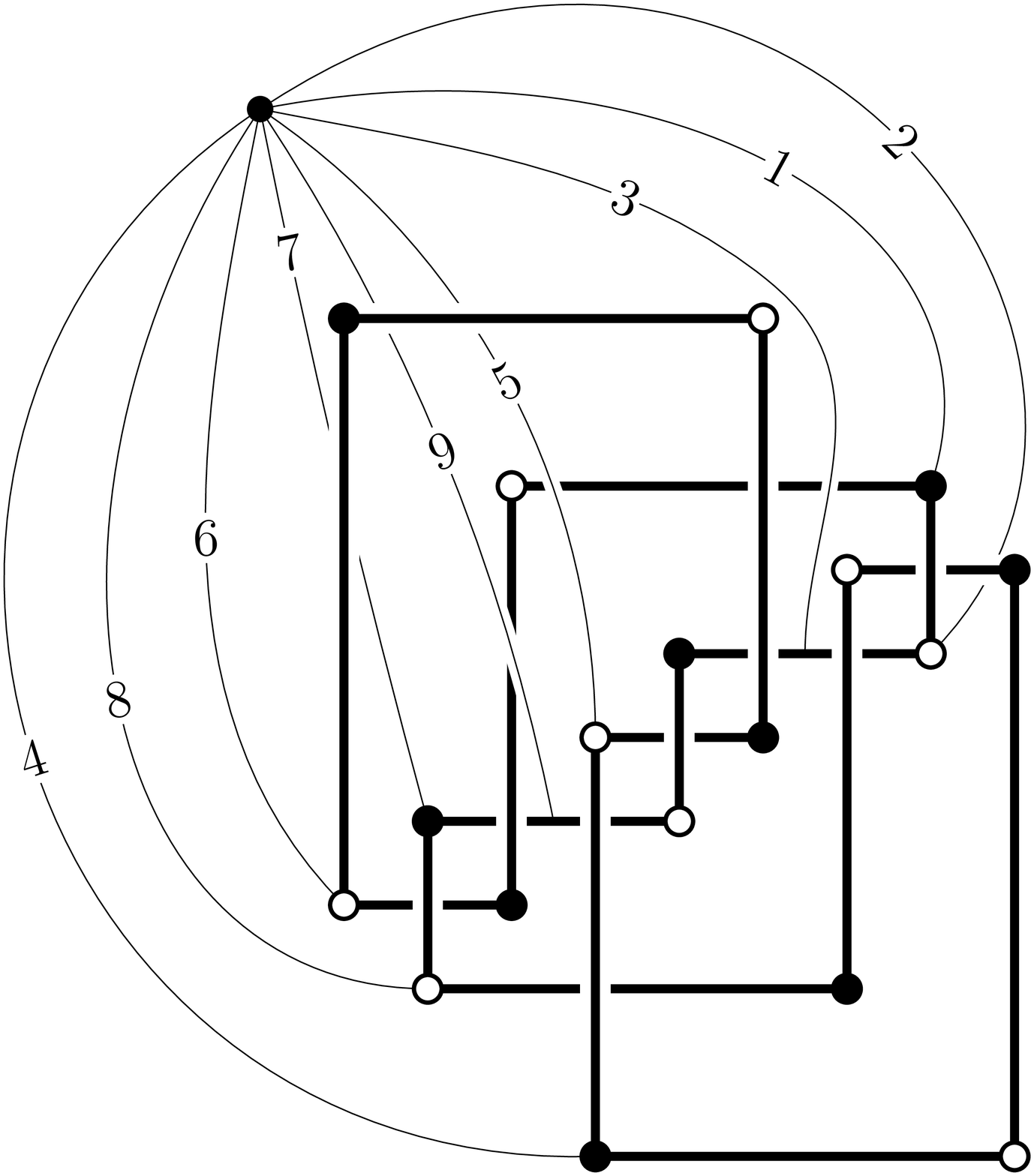}\\
\\
\\
\\
\includegraphics[scale=0.25]{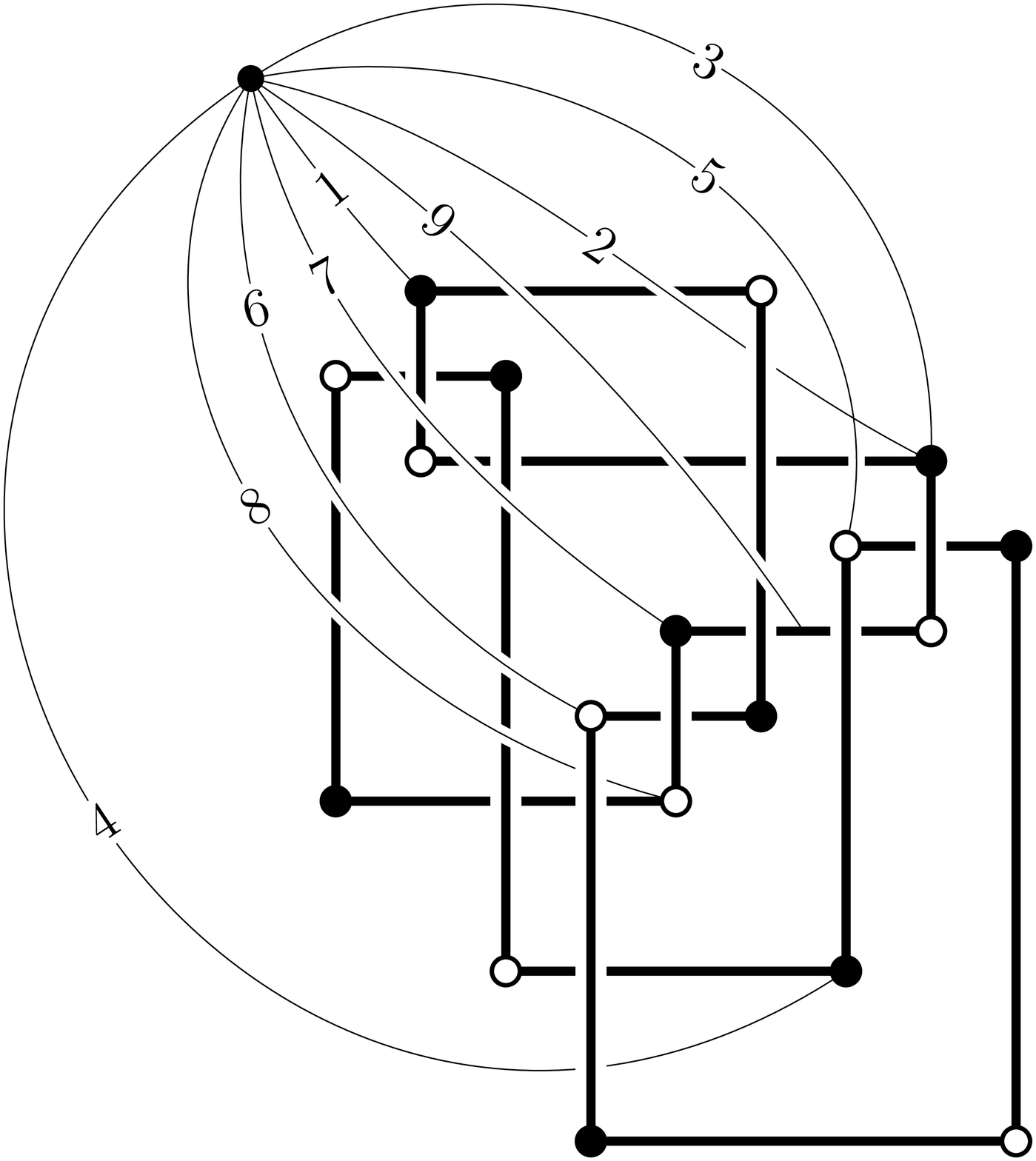}&$\quad\quad\quad$&\includegraphics[scale=0.25]{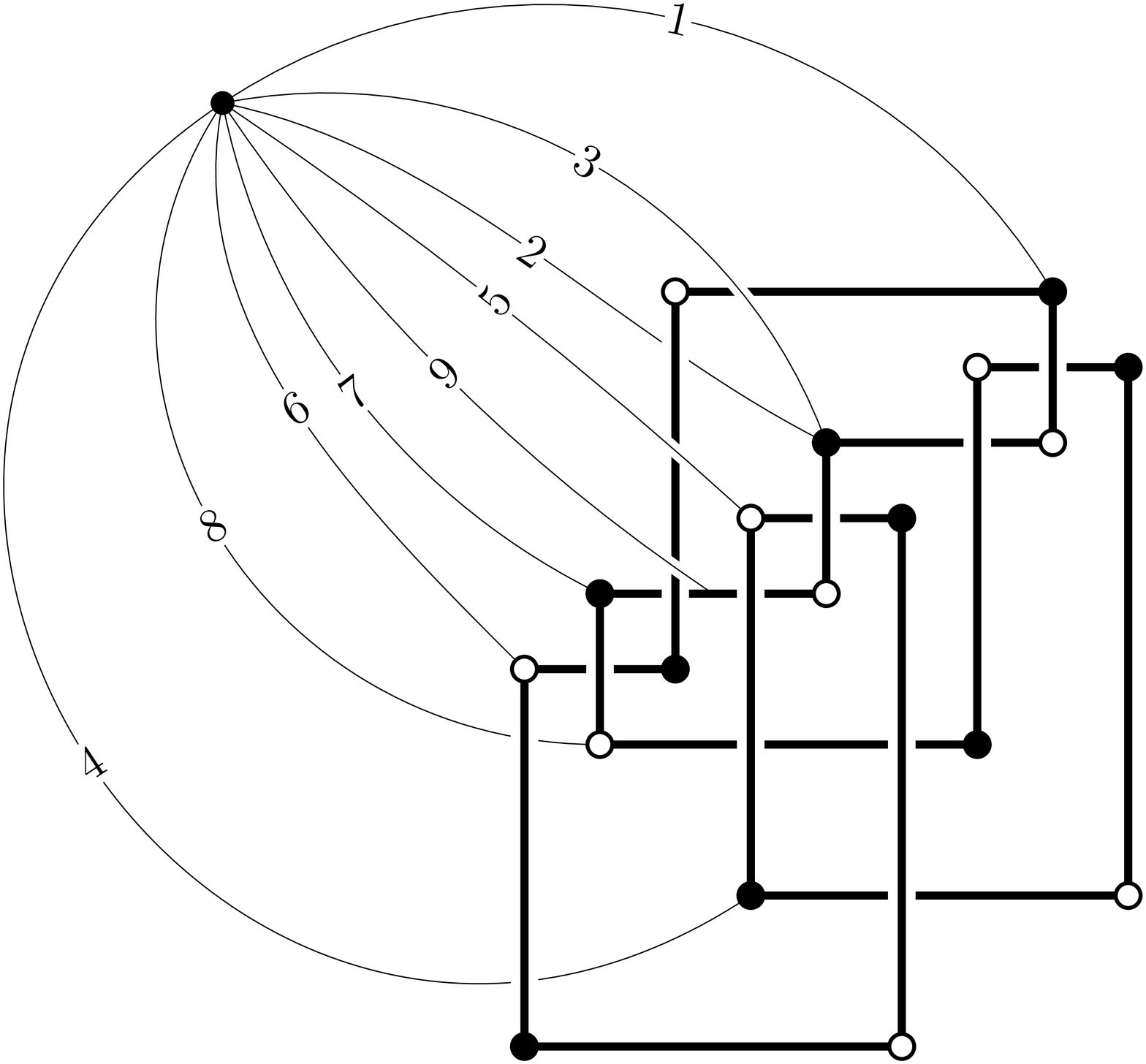}
\end{tabular}
\put(-250,120){$\longmapsto$}
\put(-246,125){$T$}
\put(-250,-120){$\longmapsto$}
\put(-350,10){\rotatebox{-90}{${\longmapsto}$}}
\put(-83,-5){\rotatebox{90}{${\longmapsto}$}}
\caption{The action of $T$ on the fundamental group}
\label{9_48loops}
\end{figure}

The set of defining relations for the Wirtinger presentation in terms of these generators:
\begin{multline}\label{Wirtinger9_48}
\begin{split}
x_1x_2x_1^{-1} &= x_3,\quad x_4x_2x_4^{-1} = x_1, \quad x_2x_4x_2^{-1} = x_8, \quad x_6x_8x_6^{-1} = x_7,\quad x_7x_1x_7^{-1} = x_6, \\ x_1x_7x_1^{-1} &= x_9, \quad x_5x_9x_5^{-1} = x_3, \quad x_3x_6x_3^{-1} = x_5, \quad x_6x_4x_6^{-1} = x_5.
\end{split}
\end{multline}

\begin{rema}\label{gener}
It is clear from these relations, that $G(9_{48})$ is generated by the elements $x_1,x_2,x_7.$
\end{rema}

%
%
%

The corresponding action of $T_\ast$ on $G(9_{48})$ is the following:
\begin{multline}\label{Action9_48_1}
\begin{split}
T_\ast(x_1) &= x_3, \quad T_\ast(x_2) = x_6^{-1}x_3x_6, \quad T_\ast(x_3) = x_9, \quad T_\ast(x_4) = x_6, \\ T_\ast(x_5) &= x_1x_6x_1^{-1}, \quad T_\ast(x_6) = x_1, \quad T_\ast(x_7) = x_1x_4x_1^{-1}, \quad T_\ast(x_8) = x_4, \quad T_\ast(x_9) = x_1x_8x_1^{-1}.
\end{split}	
\end{multline}

From~\eqref{Wirtinger9_48} and~\eqref{Action9_48_1} we obtain the action  $T^3_\ast$ on $G(9_{48})$
\begin{multline}\label{Action9_48_3}
\begin{split}
T_\ast^3(x_1) &= x_1x_8x_1^{-1}, \quad T_\ast^3(x_2) = T_\ast^3(x_4^{-1}x_1x_4) = x_3^{-1}x_1x_8x_1^{-1}x_3 = x_1x_2^{-1}x_8x_2x_1^{-1} =x_1x_4x_1^{-1}, \\ T_\ast^3(x_3) &= x_3x_4x_3^{-1}, \quad T_\ast^3(x_4) = x_3, \quad T_\ast^3(x_5) = T_\ast^3(x_6x_4x_6^{-1}) =  x_9x_3x_9^{-1}, \quad T_\ast^3(x_6) = x_9, \\ T_\ast^3(x_7) &=x_9x_1x_9^{-1} = x_1x_7x_1x_7^{-1}x_1^{-1} = x_1x_6x_1^{-1}, \quad T_\ast^3(x_8) = x_1, \quad T_\ast^3(x_9) = x_9x_6x_9^{-1}.
\end{split}	
\end{multline}

It follows that
\begin{multline}\label{Action9_48_6}
\begin{split}
T_\ast^6(x_1) &= (x_1x_8)x_1(x_1x_8)^{-1}, \quad T_\ast^6(x_2) = x_1x_8x_1^{-1}x_3x_1x_8^{-1}x_1^{-1} = (x_1x_8)x_2(x_1x_8)^{-1}, \\  T_\ast^6(x_7) &=x_1x_8x_1^{-1}x_9x_1x_8^{-1}x_1^{-1} = (x_1x_8)x_7(x_1x_8)^{-1}.
\end{split}	
\end{multline}

Therefore the action of $T_\ast^6$ on $G(9_{48})$ is the conjugation by $x_1x_8$ by remark~\ref{gener}.


%


If the action of $T_\ast^2$ on $G(9_{48})$ is the conjugation by an element $z$, then $z^3 = x_1x_8$, because the center of $G(9_{48})$ is trivial. But the linking number of $z^3$ with $\widehat{9_{48}^5}$  is divisible by $3$, while the linking number of $x_1x_8$ equals 2.

Denote by $S$ the following automorphism of $G(9_{48})$:

\begin{equation}
S(x) = x_1^{-1}T_\ast^3(x)x_1.
\end{equation}

From~\eqref{Action9_48_3} we obtain the action of $S$ on $G(9_{48})$

\begin{multline}\label{ActionS}
\begin{split}
S(x_1) &= x_8, \quad S(x_2) = x_4, \quad S(x_3) = x_1^{-1}x_3x_4x_3^{-1}x_1, \quad S(x_4) = x_2, \\ S(x_5) &= x_1^{-1}x_9x_3x_9^{-1}x_1, \quad S(x_6) = x_7, \quad S(x_7) = x_6, \quad S(x_8) = x_1, \quad S(x_9) = x_1^{-1}x_9x_6x_9^{-1}x_1
\end{split}	
\end{multline}
and that $S^2$ is the trivial automorphism of $G(9_{48})$. It is sufficient to check that $S^2$ is the identity for the generators $x_1, x_2, x_7$, see Remark~\ref{gener}.

If the action of $S$ on $G(9_{48})$ is the conjugation by an element $w$, then $w^2 = 1$, because the center of $G(9_{48})$ is trivial. The torsion group of the knot group is trivial (see for example~\cite[Chapter~9]{hemp}) so $w = 1$.
If $S$ is trivial, then 
$x_1 = x_8$, $x_2 = x_4$ and $x_6 = x_7$. It follows from the relations~(\ref{Wirtinger9_48}) that all generators in this case are equal and $G(9_{48}) = \mathbb{Z}$.  By Dehn's Lemma (see~\cite{papa}) we obtain a contradiction because the knot $9_{48}$ is nontrivial.


\end{proof}

\begin{figure}[h]
\includegraphics[scale=.2]{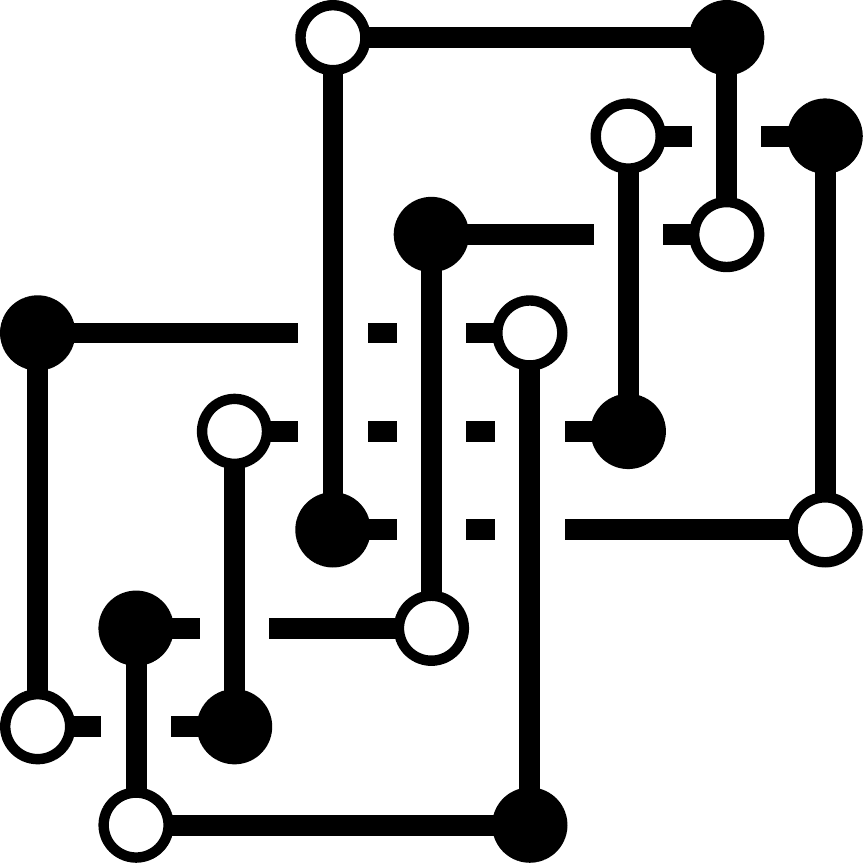}
\put(-35,-15){$10_{136}^5$}
\caption{A symmetric diagram having topological type $10_{136}$}
\label{10_136-symmetric}
\end{figure}

\begin{lemm}
\label{10_136-symmetry}
Let $10_{136}^5$ be the diagram in Figure~\ref{10_136-symmetric}. Then 
$\mathrm{Sym}_-(10_{136}^5) = \mathrm{Sym}(10_{136}^5)$.
\end{lemm}

\begin{proof}
It is shown in \cite{boizi} that the knot $10_{136}$ is nonelliptic Montesinos knot and its symmetry group is~$\mathbb Z_2$.
Since the homeomorphism $\widehat r_{\scriptscriptstyle\diagdown}$ is of finite order and the knot $10_{136}$ is not a torus knot, $[\widehat r_{\scriptscriptstyle\diagdown}]_{\widehat{10_{136}^5}}$ is the unique nontrivial element of the symmetry group by Lemma~\ref{trivial-symmetry}.

By Lemma~\ref{symmetry-morphism}, $[\widehat r_{\scriptscriptstyle\diagdown}]_{\widehat{10_{136}^5}}\in\mathrm{Sym}_-(10_{136}^5)$, hence $\mathrm{Sym}_-(10_{136}^5) = \mathrm{Sym}(10_{136}^5)$. The proof of Lemma~\ref{symmetry-morphism} is constructive and the reader can easily extract from the proof a sequence of elementary moves representing the morphism.

\end{proof}

\begin{rema}
Although in our sense all elements of the symmetry group are assumed to preserve the orientation of the sphere and of each component of the link, we just note that the rotation by $180^{\circ}$ of the diagrams $7_4^1,$ $9_{48}^5$ and $10_{136}^5$ provides a homeomorphism that preserves the orientation of the sphere and reverses the orientation of the knot.
\end{rema}

\section{Distinguishing}
\label{distinguishing}

\begin{lemm}
\label{7_4-xi-minus}
$\mathscr L_-(7_4^1) = \mathscr L_-(7_4^2)$.
\end{lemm}
\begin{proof}
All rectangular diagrams in the statement of the lemma are presented in the atlas~\cite{chong2013}. So it is known that they belong to the same topological type.

All $\xi_-$-Legendrian types in the statement of this lemma have the same pair of Thurston-Bennequin and rotation numbers. In this topological type it is known (see~\cite{chong2013}) that there exists only one $\xi_-$-Legendrian type having this pair of classical invariants. Hence the lemma.

We also provide ancillary files~\cite{anc} containing sequences of exchange moves and (de)stabilizations of type~II connecting the rectangular diagrams in the statement of this lemma.
\end{proof}

The following two lemmas have exactly the same proof as the previous one.

\begin{lemm}
\label{9_48-xi-minus}
$\mathscr L_-(9_{48}^1) = \mathscr L_-(9_{48}^3) = \mathscr L_-(9_{48}^4) = \mathscr L_-(9_{48}^5).$
\end{lemm}

\begin{lemm}
\label{10_136-xi-minus}
$\mathscr L_-(10_{136}^1) = \mathscr L_-(10_{136}^2) = \mathscr L_-(10_{136}^3) = \mathscr L_-(10_{136}^4) = \mathscr L_-(10_{136}^5).$
\end{lemm}

\begin{proof}[Proof of Proposition~\ref{main-result-7_4}]
No exchange move can be applied to the diagram $7_4^1.$ Also one can check that the diagrams $7_4^1$ and $7_4^2$ are combinatorially nonequivalent.

By Lemma~\ref{7_4-xi-minus}, $\mathscr L_-(7_4^1) = \mathscr L_-(7_4^2)$. Since $[7_4^1] \neq [7_4^2]$, by Lemma~\ref{7_4-symmetry} and Theorem~\ref{exchange-count} we have $\mathscr L_+(7_4^1) \neq \mathscr L_+(7_4^2).$
\end{proof}

The following lemma is a straight-forward consequence of the definitions.

\begin{lemm}\label{symmetry-conjugacy}
If $\eta$ is a morphism from $\widehat R_1$ to $\widehat R_2$ which can be decomposed into morphisms associated with exchange moves and type~II (de)stabilizations then the map $\psi \mapsto \eta\psi\eta^{-1}$ induces an isomorphism of pairs $(\mathrm{Sym}(R_1), \mathrm{Sym}_-(R_1))$ and $(\mathrm{Sym}(R_2), \mathrm{Sym}_-(R_2)).$
\end{lemm}

\begin{proof}[Proof of Proposition~\ref{main-result-10_136}]
By Lemmas~\ref{10_136-symmetry},~\ref{10_136-xi-minus}~and~\ref{symmetry-conjugacy}, $\mathrm{Sym}_-(10_{136}^1)= \mathrm{Sym}(10_{136}^1)$ and $\mathrm{Sym}_-(10_{136}^3)= \mathrm{Sym}(10_{136}^3).$


It is easy to check that the exchange class $[10_{136}^1]$ consists of one element and the exchange class $[10_{136}^3]$ consists of three elements. And also that $10_{136}^2\notin[10_{136}^1]$ and $10_{136}^4\notin[10_{136}^3].$

By Lemma~\ref{10_136-xi-minus}, $\mathscr L_-(10_{136}^1) = \mathscr L_-(10_{136}^2).$  So by Theorem~\ref{exchange-count}, the conditions $\mathrm{Sym}_-(10_{136}^1)= \mathrm{Sym}(10_{136}^1)$, $\mathscr L_-(10_{136}^1) = \mathscr L_-(10_{136}^2)$ and $10_{136}^2\notin[10_{136}^1]$ imply $\mathscr L_+(10_{136}^1) \neq \mathscr L_+(10_{136}^2).$ The same proof holds for the inequality $\mathscr L_+(10_{136}^3) \neq \mathscr L_+(10_{136}^4).$
\end{proof}

\begin{proof}[Proof of Proposition~\ref{main-result-9_48}]
By Lemmas~\ref{9_48-symmetry},~\ref{9_48-xi-minus}~and~\ref{symmetry-conjugacy}, $\mathrm{Sym}_-(9_{48}^1)= \mathrm{Sym}(9_{48}^1)$ and $\mathrm{Sym}_-(9_{48}^3)= \mathrm{Sym}(9_{48}^3)$. 

The exchange class $[9_{48}^3]$ consists of three elements and it is a direct check that $9_{48}^4\notin[9_{48}^3].$ Since $\mathscr L_-(9_{48}^3) = \mathscr L_-(9_{48}^4)$ (by Lemma~\ref{9_48-xi-minus}), $9_{48}^4\notin[9_{48}^3]$ and $\mathrm{Sym}_-(9_{48}^3)= \mathrm{Sym}(9_{48}^3)$, then $\mathscr L_+(9_{48}^3) \neq \mathscr L_+(9_{48}^4)$ by Theorem~\ref{exchange-count}.

Note that the diagram $9_{48}^2$ is obtained from the diagram $9_{48}^1$ by reversing the orientation and their $\xi_-$-Legendrian types have opposite nonzero rotation numbers. So $\mathscr L_-(9_{48}^1)\neq \mathscr L_-(9_{48}^2)$ and we must argue in the different way to distinguish Legendrian types $\mathscr L_+(9_{48}^1)$ and $\mathscr L_+(9_{48}^2)$.

Denote the diagram on the right of Figure~\ref{9_48-1to2} by $9_{48}^6$. We will show that $\mathscr L_-(9_{48}^1) = \mathscr L_-(9_{48}^6)$, $\mathscr L_+(9_{48}^6) = \mathscr L_+(9_{48}^2)$ and $\mathscr L_+(9_{48}^1) \neq \mathscr L_+(9_{48}^6)$. Hence $\mathscr L_+(9_{48}^1) \neq \mathscr L_+(9_{48}^2).$

\begin{figure}
\includegraphics[scale=.2]{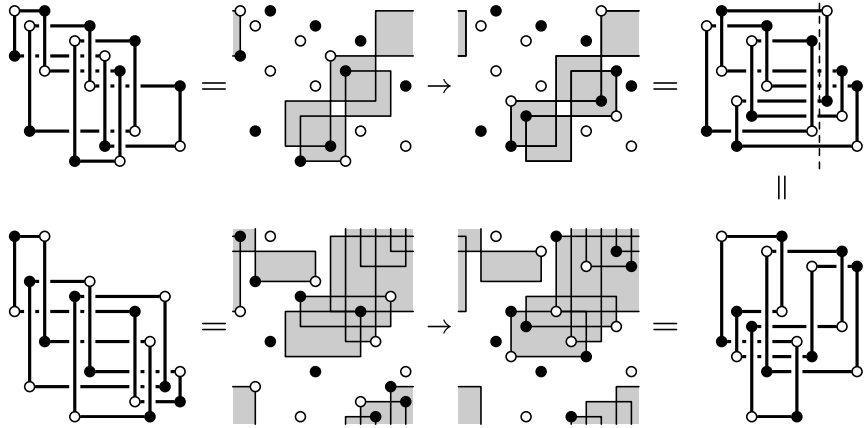}
\caption{A transitions from $9_{48}^1$ ($9_{48}^2$) to $9_{48}^6$ by moves not including stabilizations and destabilizations of type~I (type~II)}
\label{9_48-1to2}
\end{figure}

The top row of Figure~\ref{9_48-1to2} illustrates a sequence of exchange moves and (de)stabilizations of type~II from the diagram $9_{48}^1$ to the diagram $9_{48}^6$. Each elementary move is associated with some rectangle in this figure. We apply these moves in such an order that at each step the corresponding rectangle provides an applicable elementary move. Although there are more than one such orderings and in different orderings with the same rectangle different elementary moves can be associated, in each case we obtain a sequence of elementary moves from the diagram $9_{48}^1$ to the diagram $9_{48}^6$. See Figure~\ref{9_48-1to6} for an example of such a sequence for some ordering. For convenience, the vertices that are being eliminated or moving we paint red.

\begin{figure}
\includegraphics[scale=.2]{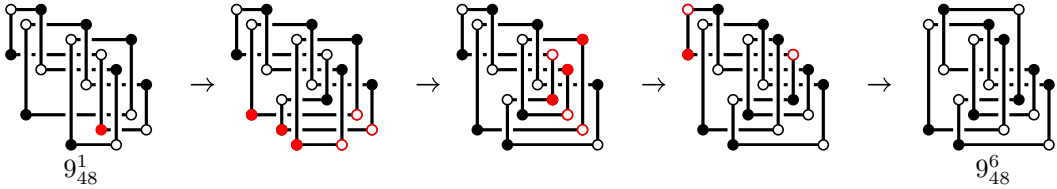}
\put(-375,-10){$9_{48}^1$}
\put(-30,-10){$9_{48}^6$}
\caption{A transition from $9_{48}^1$ to $9_{48}^6$ by a stabilization of type~II, two exchanges of horizontal edges, two exchanges of vertical edges and a destabilization of type~II}
\label{9_48-1to6}
\end{figure}

Similarly the bottom row of Figure~\ref{9_48-1to2} illustrates a sequence of exchange moves and (de)stabilizations of type~I from the diagram $9_{48}^2$ to the diagram $9_{48}^6$. The top diagram in the right column of Figure~\ref{9_48-1to2} is mapped to the bottom diagram in this column by a combinatorial equivalence which maps a dashed line to the vertical line $\theta = 0$.

\begin{figure}[h]
\begin{tabular}{ccc}
\includegraphics[scale=.2]{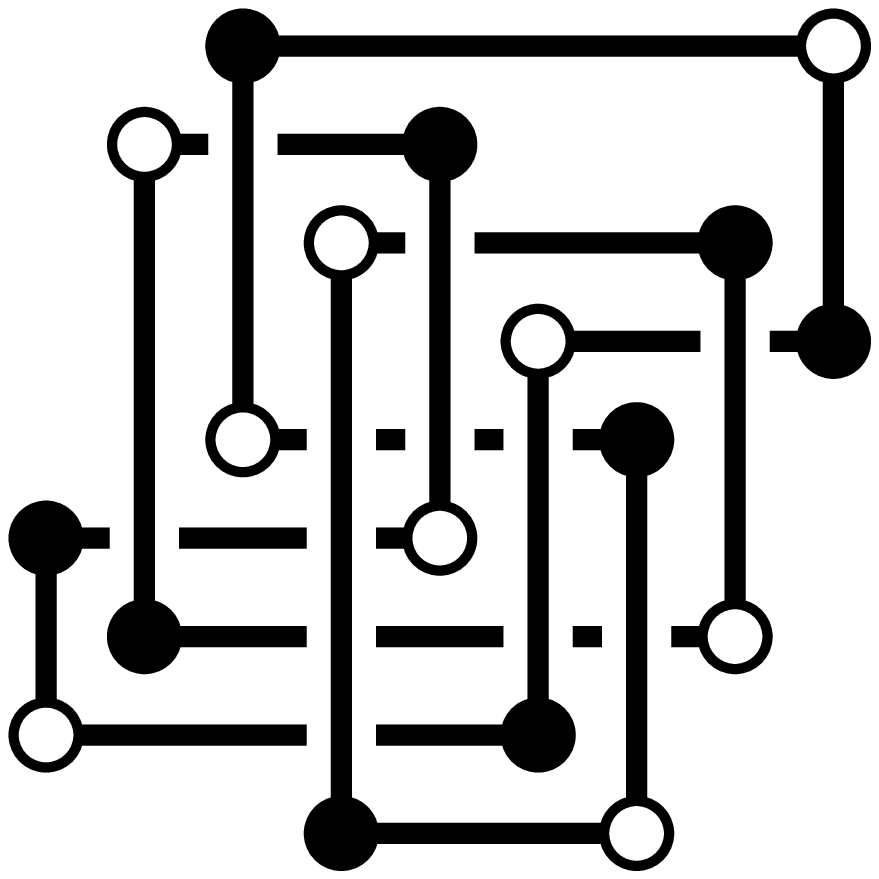}&\includegraphics[scale=.2]{9_48-1.eps}&\\
\includegraphics[scale=.2]{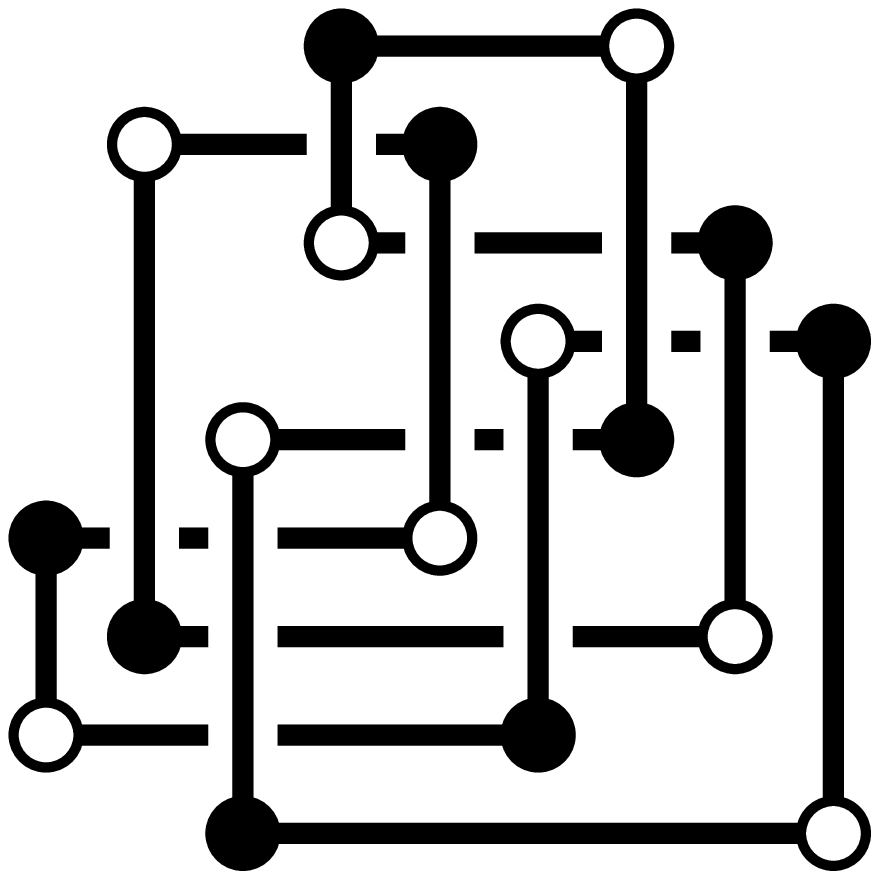}&\includegraphics[scale=.2]{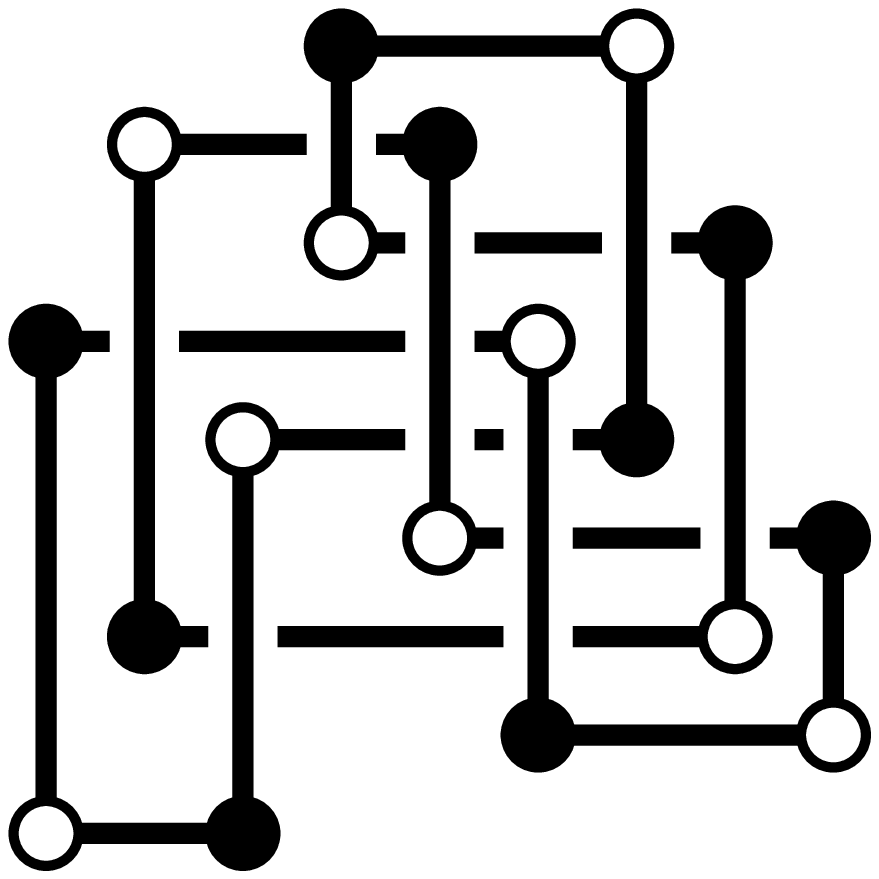} & \includegraphics[scale=.2]{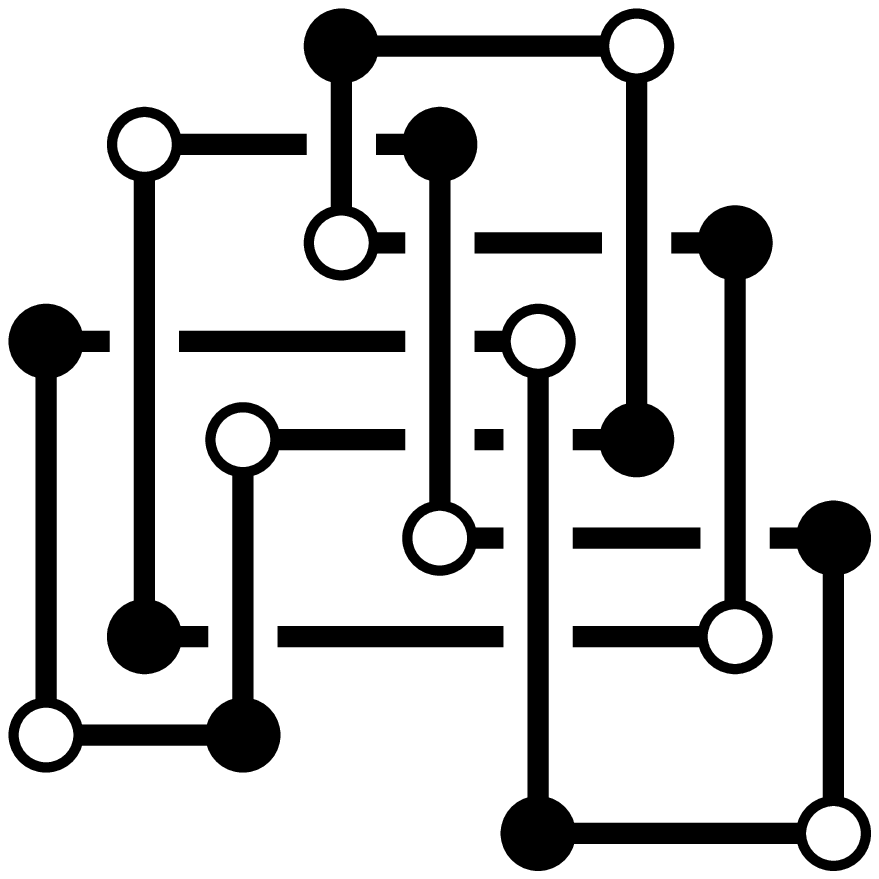} \\
\includegraphics[scale=.2]{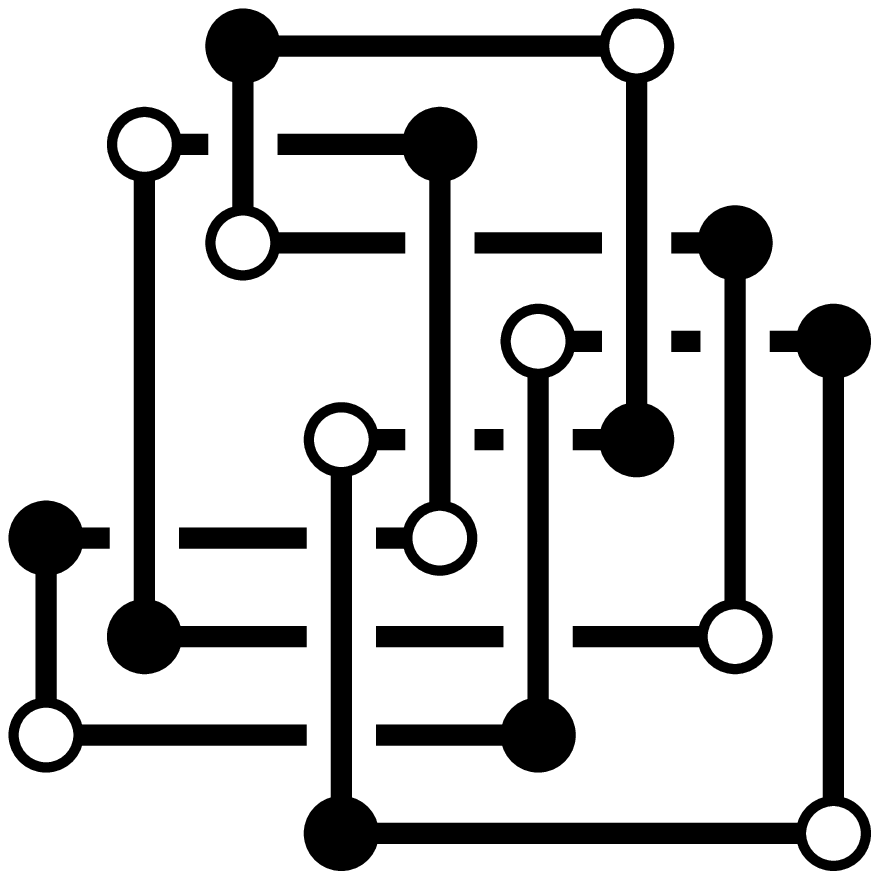}&\includegraphics[scale=.2]{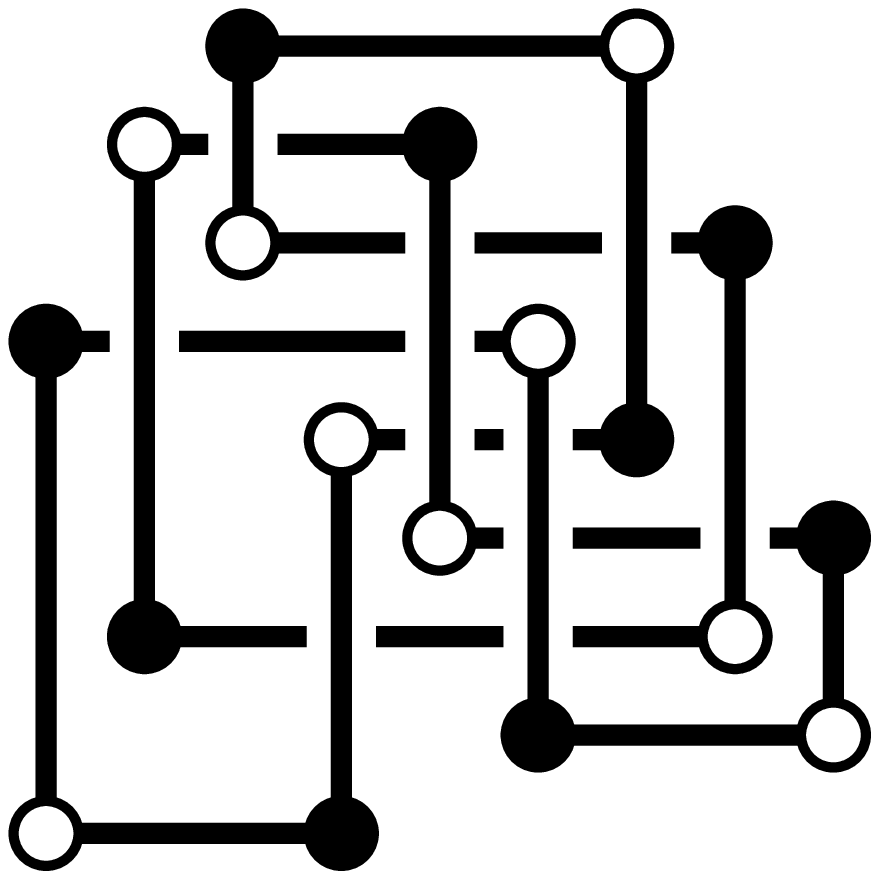} & \includegraphics[scale=.2]{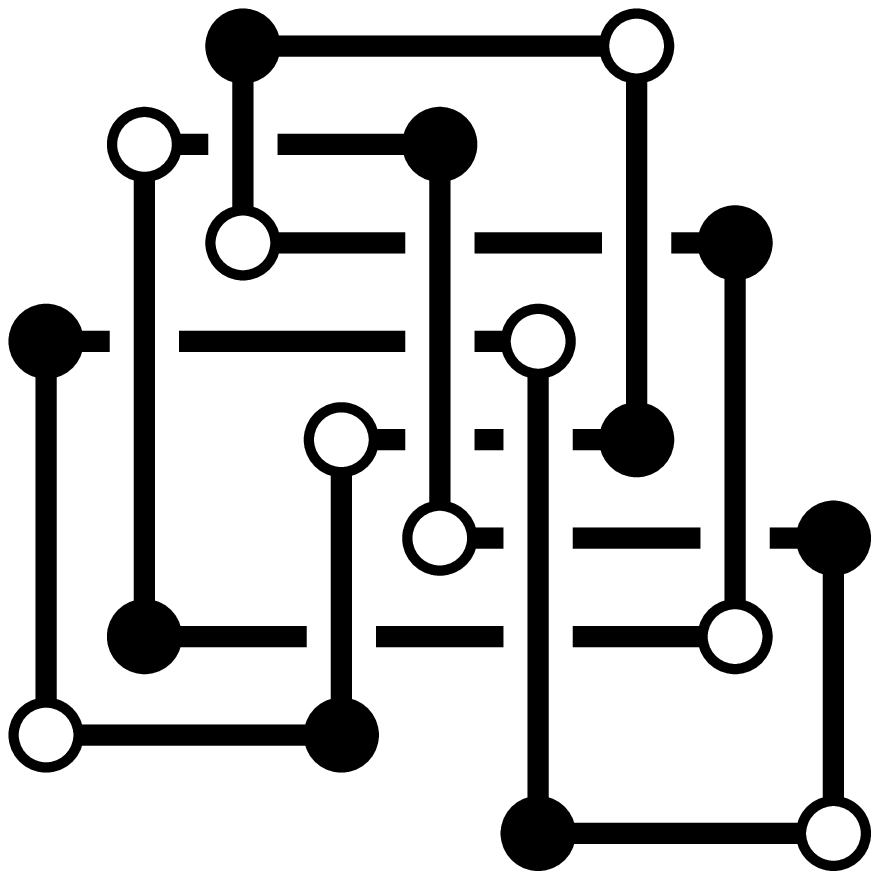}
\end{tabular}
\caption{Exchange class of $9_{48}^1$}
\label{9_48-1-exchange}
\end{figure}

Now let us prove that $\mathscr L_+(9_{48}^1) \neq \mathscr L_+(9_{48}^6).$

The whole exchange class of $9_{48}^1$ is presented in Figure~\ref{9_48-1-exchange}. Two classes of combinatorial equivalence are related by an exchange move if and only if they are neighbors in this table (lie in a common row and adjacent columns or lie in a common column and adjacent rows). For each diagram $R$ in Figure~\ref{9_48-1-exchange} there exists a pair of horizontal levels $\varphi_0$ and $\varphi_1$ such that
\begin{enumerate}
\item there exists a vertical edge of the diagram $R$ with endpoints on the levels $\varphi_0$ and $\varphi_1$;
\item there exists exactly one occupied horizontal level $\varphi_2\in(\varphi_0;\varphi_1)$ of the diagram $R$.
\end{enumerate} 

There is no such pair of horizontal levels for the diagram $9_{48}^6.$ Hence $9_{48}^6\notin[9_{48}^1].$ Finally $[9_{48}^1]\neq[9_{48}^6],$ $\mathscr L_-(9_{48}^1) = \mathscr L_-(9_{48}^6)$ and $\mathrm{Sym}_-(9_{48}^1)= \mathrm{Sym}(9_{48}^1)$ imply $\mathscr L_+(9_{48}^1) \neq \mathscr L_+(9_{48}^6)$ by Theorem~\ref{exchange-count}.

\end{proof}

\section{Lemmas on morphisms}
\begin{lemm}[see \cite{gif67} and "Borel's theorem" in \cite{bosie}]
\label{trivial-symmetry}
Let $L$ be a link in $\mathbb S^3$, $h$ be a self-homeomorphism of a pair $(\mathbb S^3, L)$ and $r$ be a positive integer such that
\begin{enumerate}
\item $h\neq \mathrm{id}_{\mathbb S^3}$;
\item $h^r = \mathrm{id}_{\mathbb S^3}$;
\item $h$ is isotopic to $\mathrm{id}_{\mathbb S^3}$ through self-homeomorphisms of the pair $(\mathbb S^3, L).$
\end{enumerate}
Then the center of $\pi_1(\mathbb S^3\setminus L)$ is nontrivial or $L$ is the trivial link.

The category is assumed to be PL or Diff.
\end{lemm}
\begin{proof}

Let $h_t$ be an isotopy such that $h_0 = \mathrm{id}_{\mathbb S^3}$, $h_1 = h$ and $h_t(L)\subset L$ for any $t\in[0;1].$ Let $p\in\mathbb S^3\setminus L$ and let $\alpha$ be an arc $h_t(p)$ where $t\in[0;1].$ Let $\gamma_p$ be the loop 
$$
\alpha \circ h\alpha \circ \dots \circ h^{r-1}\alpha.
$$

Let us prove that $\gamma_p$ lies in the center of $\pi_1(\mathbb S^3\setminus L, p)$ for any $p.$

Let $\gamma'(t),$ $t\in[0;1]$ be any loop with $\gamma'(0)=\gamma'(1)=p$. Consider a homotopy between arcs $\gamma'\circ\alpha$ and $\alpha\circ h\gamma'$ fixing endpoints:
$$
\Upsilon_{t,s} = 
\left\{
\begin{aligned}
\alpha(3ts)&, \quad t\in[0;1/3],\\
h_s(\gamma'(3t-1))&, \quad  t\in[1/3;2/3],\\
\alpha(s+(1-s)(3t-2))&, \quad  t\in[2/3;1],
\end{aligned}
\right.
$$
where $s\in[0;1].$

Composing this homotopy with the homeomorphisms $h^k$ for $k = 1,\dots,r-1$ we get
\begin{multline*}
\gamma'\circ\gamma_p = \gamma'\circ \alpha \circ h\alpha \circ \dots \circ h^{r-1}\alpha \sim \alpha\circ h\gamma' \circ h\alpha \circ \dots \circ  h^{r-1}\alpha \sim \dots \\ \dots \sim \alpha \circ h\alpha \circ\dots\circ h^{r-1}\gamma' \circ h^{r-1}\alpha \sim \alpha \circ h\alpha \circ \dots \circ h^{r-1}\alpha \circ \gamma' = \gamma_p\circ\gamma'.
\end{multline*}

Hence $\gamma_p$ lies in the center.

Suppose that $\gamma_p$ is always trivial.

Since $h$ is isotopic to the identity it preserves the orientation of the sphere, and orientations and numbering of connected components of $L$. Consider some connected component $c\subset L$. 

\smallskip
\noindent\emph{Case 1.} $h$ acts nontrivially on $c.$

In this case the action of $h$ on $c$ is conjugate to the rotation by $2\pi k/r$ of the standard circle $\mathbb S^1$ for some $k \in \{1, 2, \dots, r-1\}$. Choose a tubular neighborhood $U\cong\mathbb S^1\times \mathbb D^2$ of $c$ such that if $h(\theta, x) = (\theta', x')$ then $|\theta'-\theta-2\pi k/r|<\pi/r$ for any $\theta\in\mathbb S^1$ and $x\in\mathbb D^2$. Since the isotopy keeps $c$ invariant, we can choose a point $p\in U\setminus c$ such that $\gamma_p\subset U$. By construction, $\gamma_p$ is free homotopic in $U$ to $l$ times the class $[c]$ where $l \equiv k$ $\mathrm{mod}$ $r$.

This means that if $\gamma_p$ is trivial in $\pi_1(\mathbb S^3\setminus L)$ then some curve on $\partial U$ is trivial in the same sense. By Dehn's Lemma, some curve which is not trivial on $\partial U$ bounds a disk in $\mathbb S^3\setminus L.$ This is possible only for the longitude because the other simple closed curves on $\partial U$ are homologically nontrivial in $\mathbb S^3\setminus L$. Thus the component $c$ bounds a disk not intersecting the other components.

\smallskip
\noindent\emph{Case 2.} $h$ acts trivially on $c.$

Suppose that the category is Diff. By~\cite[Theorem 2.2, p.~306]{bre72} there exists a neighborhood $U\cong\mathbb S^1\times\mathbb D^2$ of $c$ such that each circle $q\times \mathbb S^1_{\rho}$ is invariant under $h$ for any $q\in c$ and $\rho\in(0;\varepsilon)$ --- a radius of the circle, and each radius of each disk $q\times \mathbb D^2$ is mapped again to a radius of this disk under $h$. Moreover, we can choose coordinates on $\mathbb S^1\times\mathbb D^2$ such that $h$ acts on each disk by rotation by $2\pi k/r$ for some $k \in \{0, 1, \dots, r-1\}$. If $k=0$ then $h$ is identical everywhere in $\mathbb S^3\setminus L$ since the latter is connected. So $k\neq0.$ As in the previous case take $p\in U\setminus c$ sufficiently close to $c$ such that $\gamma_p\subset U.$ It is clear that $\gamma_p$ is homologous in $U\setminus c$ to $k[\mu]+r(m[\mu]+l[\lambda])$ where $m$ and $l$ are integers, $\mu$ and $\lambda$ are the meridian and the longitude of $\partial U$. Hence $\gamma_p$ is nontrivial.

In the PL case we can assume that $h$ is a simplicial map since $h$ is of finite order. Take a second baricentric subdivision of the corresponding triangulation. Then the union $U$ of simplices intersecting $c$ is a tubular neighborhood invariant under $h$. Then argue as in the Diff case.

\end{proof}

For a pair $(f,g)$ of orientation-preserving self-homeomorphisms of $\mathbb S^1$ we write $(f,g)(R)$ for a rectangular diagram whose vertices are obtained from the set of all vertices of the diagram $R$  by applying the self-homeomorphism $(f,g)$ of $\mathbb T^2$ preserving colors of vertices and numbering of connected components. 

\begin{lemm}
\label{combinatorial-morphism}
For every pair of orientation-preserving homeomorphisms $f,g:\mathbb S^1\to \mathbb S^1$ and any rectangular diagram $R$ of a link the morphism $[\widehat{(f,g)}]_{\widehat{R}}$ can be decomposed into morphisms associated with exchange moves.
\end{lemm}
\begin{proof}
Since $(f,g) = (f,\mathrm{id})\circ(\mathrm{id},g)$, it is sufficient to consider two cases: $f=\mathrm{id}$ or $g=\mathrm{id}.$ The cases are similar so we assume $g=\mathrm{id}$.

Any orientation-preserving homeomorphism of $\mathbb S^1$ can be decomposed into homeomorphisms $f_i$ such that each $f_i$ is identical outside some interval which contains zero or one vertical level of $R$. So we can assume the latter for $f$.

Suppose $f$ is identical outside the interval $(a;b).$

If $(a;b)$ does not contain any of the vertical levels of $R$ then the homeomorphism $\widehat{(f,\mathrm{id})}:(\mathbb S^3, \widehat R)\to(\mathbb S^3, \widehat{R})$ is isotopic to the identity, and there is nothing to prove.

Suppose that there is only one vertical level $\theta_i$ of $R$ such that $\theta_i\in (a;b).$ By passing to $f^{-1}$ if necessary we can assume that $f(\theta_i)\in (\theta_i;b).$ Let $(\theta_i, \varphi_0)$ and $(\theta_i, \varphi_1)$ be two vertices of $R$ lying on the vertical level $\theta_i.$ Let $d$ be the set of points $\{(\theta,\varphi,\tau)\}$ where $\theta\in[\theta_i; f(\theta_i)],$ $\varphi\in\{\varphi_0, \varphi_1\},$ $\tau\in[0;1].$ Clearly $d$ is a disk. Let $B$ be the set of points $\{(\theta, \varphi, \tau)\}$ where $\theta\in(a;b),$ $\varphi\in\mathbb S^1,$ $\tau\in(0;1]$. Clearly $B$ is an open ball.

The homeomorphism $\widehat{(f,\mathrm{id})}$ is identical outside the ball $B$, the  ball $B$ contains the interior of $d$ and intersects $\widehat{R}\cup\widehat{(f,\mathrm{id})(R)}$ in the union of two open arcs $\{\theta_i\}\times\{\varphi_0,\varphi_1\}\times(0;1]\cup \{f(\theta_i)\}\times\{\varphi_0,\varphi_1\}\times(0;1].$ This means that the triple $(\widehat{R},\widehat{(f,\mathrm{id})(R)},[\widehat{(f,\mathrm{id})}]_{\widehat{R}})$ is a $\mathbb D^2$-move associated with the disk $d$.

Since $d$ is contained in the ball $[\theta_i; f(\theta_i)]*[\varphi_0;\varphi_1]$ we see that $\widehat{(f,\mathrm{id})}$ represents the morphism associated with an exchange move associated with the rectangle $[\theta_i; f(\theta_i)]\times[\varphi_0;\varphi_1]$.
\end{proof}

We write $r_{\scriptscriptstyle\diagdown}(R)$ for a rectangular diagram whose vertices are obtained by applying the homeomorphism $r_{\scriptscriptstyle\diagdown}$ to the set of all vertices of the diagram $R$ switching their color from black to white and vice versa.

\begin{lemm}
\label{symmetry-morphism}
For any rectangular diagram $R$ the morphism $[\widehat r_{\scriptscriptstyle\diagdown}]_{\widehat{R}}$ can be decomposed into morphisms associated with exchange moves and type~II (de)stabilizations.
\end{lemm}
\begin{proof}
This lemma follows from a more general statement. Since $\widehat r_{\scriptscriptstyle\diagdown}$ preserves the contact structure $\xi_-$ together with its coorientation(s), and since any $\xi_-$-contactomorphism of $\mathbb S^3$ which preserves the coorientation of $\xi_-$ is isotopic to the identity through $\xi_-$-contactomorphisms, there exists a Legendrian isotopy between $\widehat R$ and $\widehat {r_{\scriptscriptstyle\diagdown}(R)}$ realizing the same morphism. And actually for any two rectangular diagrams $R_1$ and $R_2$ any morphism provided by a $\xi_-$-Legedrian isotopy between $\widehat R_1$ and $\widehat R_2$ can be decomposed into morphisms associated with exchange moves and type~II (de)stabilizations. One can prove the general statement in the same way as in \cite{ngth,OST}. We mimic this proof in the partial case specified in the statement of this lemma. First, we construct a Legendrian isotopy. Then we tile the trace of the link by this isotopy into disks, such that the $\mathbb D^2$-move associated with each disk is associated with an elementary move of rectangular diagrams. In this way we obtain a decomposition into morphisms associated with elementary moves.

Let $f, g:\mathbb S^1\to\mathbb S^1$ be such orientation-preserving homeomorphisms that $(f,g)(R) \subset (3\pi/2;2\pi)\times (0;\pi/2).$ Consider the decomposition 
$$r_{\scriptscriptstyle\diagdown} = 
\left(r_{\scriptscriptstyle\diagdown}(f^{-1},g^{-1})r_{\scriptscriptstyle\diagdown}\right)
r_{\scriptscriptstyle\diagdown}
(f,g).$$
It is easy to see that $r_{\scriptscriptstyle\diagdown}(f^{-1},g^{-1})r_{\scriptscriptstyle\diagdown} = (f_1, g_1)$ for some orientation-preserving homeomorphisms $f_1, g_1:\mathbb S^1\to\mathbb S^1$ and $\widehat r_{\scriptscriptstyle\diagdown}\widehat{(f^{-1},g^{-1})}\widehat r_{\scriptscriptstyle\diagdown} = \widehat{(f_1, g_1)}.$ 

Since $\widehat r_{\scriptscriptstyle\diagdown} = \widehat{(f_1,g_1)}
\widehat r_{\scriptscriptstyle\diagdown}
\widehat{(f,g)}$, by Lemma~\ref{combinatorial-morphism} it is sufficient to prove the statement of this lemma for the diagram $(f,g)(R).$ So further we assume $\theta\in(3\pi/2;2\pi)$ and $\varphi\in(0;\pi/2)$ for any vertex $(\theta,\varphi)\in R$.

Let $f_2(\theta) = \theta - \pi/2$, $g_2(\varphi) = \varphi+\pi/2$, $F = (f_2,g_2)\circ r_{\scriptscriptstyle\diagdown}$ and $\widehat{F} = \widehat{(f_2,g_2)}\circ \widehat r_{\scriptscriptstyle\diagdown}$. Since $\widehat r_{\scriptscriptstyle\diagdown} = \widehat{(f_2^{-1},g_2^{-1})}\circ\widehat{F}$, by Lemma~\ref{combinatorial-morphism} it sufficient to prove that the morphism $[\widehat F]_{\widehat{R}}$ can be decomposed into morphisms associated with exchange moves and type~II (de)stabilizations.

The action of $F$ on the diagram $R$ is a flype (see \cite{dyn2003} for the definition). To decompose the flype $F$ into elementary moves we follow the approach of \cite{dysok}.

\begin{figure}[h]
\includegraphics[scale=.2]{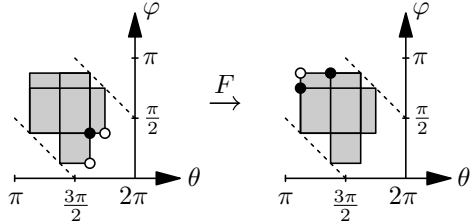}
\put(-90,35){$F$}
\put(-100,0){$\theta$}
\put(2,0){$\theta$}
\put(-46,-10){$\frac{3\pi}2$}
\put(-66,-7){$\pi$}
\put(-24,-7){$2\pi$}
\put(-15,23){$\frac{\pi}2$}
\put(-15,46){$\pi$}
\put(-15,65){$\varphi$}
\put(-148,-10){$\frac{3\pi}2$}
\put(-168,-7){$\pi$}
\put(-126,-7){$2\pi$}
\put(-117,23){$\frac{\pi}2$}
\put(-117,46){$\pi$}
\put(-117,65){$\varphi$}
\caption{Action of $F$ on vertices of a rectangular diagram}
\label{flype-example}
\end{figure}

Namely, arrange the vertices of the diagram $R$ in order of nondecreasing of $\theta - \varphi$. Following this ordering for every vertex $v_i\in R$ construct a rectangle $r_i$ whose right bottom corner is $v_i$ and whose left upper corner is $F(v_i).$ It is easy to see that the other two corners of the rectangle lie on the lines $\theta+\varphi = 2\pi\pm\pi/2$ (see Figure~\ref{flype-example}). Then apply an elementary move $s_i$ associated with the rectangle $r_i$.

\begin{figure}[h]
\includegraphics[scale=.2]{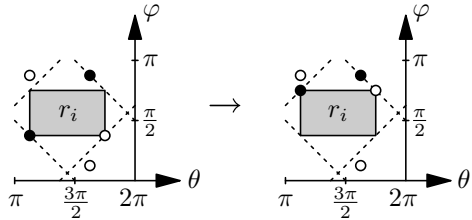}
\put(-100,0){$\theta$}
\put(-47,27){$r_i$}
\put(2,0){$\theta$}
\put(-46,-8){$\frac{3\pi}2$}
\put(-66,-7){$\pi$}
\put(-24,-7){$2\pi$}
\put(-16,23){$\frac{\pi}2$}
\put(-15,46){$\pi$}
\put(-15,65){$\varphi$}
\put(-148,-8){$\frac{3\pi}2$}
\put(-149,27){$r_i$}
\put(-168,-7){$\pi$}
\put(-126,-7){$2\pi$}
\put(-118,23){$\frac{\pi}2$}
\put(-117,46){$\pi$}
\put(-117,65){$\varphi$}
\caption{A single elementary move in the decomposition of $F$}
\label{flype-step}
\end{figure}

This ordering guarantees that at each step the corresponding elementary move will be applicable. More precisely, at the moment of eliminating the vertex $v_i (\theta_i, \varphi_i)$ the domain $D_i\subset\mathbb T^2$, where
$$
D_i:
\left\{
\begin{aligned}
\theta_i-\varphi_i-\pi&<\theta-\varphi<\theta_i-\varphi_i,\\ 
2\pi-\pi/2&<\theta+\varphi<2\pi+\pi/2,
\end{aligned}
\right.
$$
 contains no vertices of the rectangular diagram. Note that the vertices $v_i$ and $F(v_i)$ lie on the boundary of $D_i$ and the rectangle $r_i$ is contained in $D_i$ (see Figure~\ref{flype-step}). So the intersection of the rectangle $r_i$ with the set of all vertices of the rectangular diagram consists of the vertex $v_i$ and possibly one or two more vertices of $r_i$ lying on the lines $\theta+\varphi=2\pi\pm\pi/2.$

At the intersection of each occupied vertical (horizontal) level of the diagram $R$ with the line $\theta+\varphi=2\pi+\pi/2$ (respectively, $\theta+\varphi=2\pi-\pi/2$) a vertex is created and then eliminated during this process, and no other vertices appear. Therefore the sequence of elementary moves $\{s_i\}$ produces the diagram $F(R)$ from the diagram $R$:

$$
R = R_0 \stackrel{s_1}{\longmapsto} R_1 \stackrel{s_2}{\longmapsto}\dots \stackrel{s_N}{\longmapsto} R_N = F(R).
$$

Now let us prove that a composition of morphisms $\widehat{s_i}$ equals $[\widehat F]_{\widehat R}.$ We will construct an isotopy $H:\mathbb S^3\times[0;1] \to \mathbb S^3$ such that
\begin{enumerate}
\item $H(\bullet, 0) = \mathrm{id}_{\mathbb S^3}$;
\item $H(\bullet, 1) = \widehat F$;
\item $H\big|_{\widehat{v_i}\times[0;1]}$ is an embedding with image a disk $d_i$ for any $i = 1, \dots, N$;
\item $d_i \subset \Delta_{r_i}$ for any $i = 1, \dots, N$ where $\Delta_{r_i}$ is the tetrahedron associated with the rectangle $r_i$;
\item $\partial d_i = \widehat{\partial r_i}$ for any $i = 1, \dots, N$.
\end{enumerate}

By the conditions (4) and (5), $\mathbb D^2$-moves 
$$
(\widehat R_0, \widehat R_1, \widehat{s_1}), (\widehat R_1, \widehat R_2, \widehat{s_2}), \dots, (\widehat R_{N-1}, \widehat R_N, \widehat{s_N})
$$
are associated with the disks $d_1, d_2, \dots, d_N.$ By the condition~(3), the annuli $\widehat R\times[0;1]$ are tiled into the disks $\{\widehat{v_i}\times[0;1],$ $i=1,\dots,N\}$ which are mapped to the disks $\{d_i\}$ under the isotopy. Then the conditions~(1) and (2) and Lemma~\ref{isotopy-via-d2moves} imply that 
$$[\widehat F]_{\widehat R} = [H(\bullet, 1)]_{\widehat R} = \widehat{s_N}\circ\dots\circ\widehat{s_2}\circ\widehat{s_1}.$$

Now we construct such an isotopy $H$. We use the following parametrization of the unit sphere $\mathbb S^3\subset\mathbb C^2:$
$$
(\theta, \varphi, \tau) \mapsto (\cos(\pi\tau/2)e^{\mathbbm i\varphi}, \sin(\pi\tau/2)e^{\mathbbm i\theta}).
$$

Define $H((z,\overline w), \frac{t}{\pi/2}) = (z', \overline {w'}),$ where

$$
\begin{pmatrix}
z' \\
w'
\end{pmatrix}
=
\begin{pmatrix}
\cos t & \mathbbm i\sin t\\
\mathbbm i\sin t & \cos t
\end{pmatrix}
\begin{pmatrix}
z \\
w
\end{pmatrix}
\text{ for }
z, w\in\mathbb C.
$$

Since this matrix belongs to $\mathrm{SU}(2)$, $H$ is an isotopy. Moreover, if we allow $t$ to be any real number then these matrices form a subgroup. Next we check the conditions above one by one.

\noindent{\it Condition (1)}. If $t=0$, we get the identity matrix.

\noindent{\it Condition (2)}. If $t=\pi/2$, we get $z' = \mathbbm iw$ and $w' = \mathbbm iz$. Let $\theta', \varphi', \tau'$ be the coordinates of the image of the point $(\theta, \varphi, \tau).$ Since 
$$
\theta = -\arg w,\ \ \varphi = \arg z,\ \ \tau = \frac2{\pi}\arctan \left|\frac{w}z\right|
$$
and similar formulas hold for variables with primes we see that
$$
\theta' = -\frac{\pi}2 - \varphi, \ \ \varphi' = \frac{\pi}2 -\theta, \ \ \tau' = 1 - \tau. 
$$

\noindent{\it Condition (3)}. Since the matrix is unitary and $(1,1)$ is an eigenvector, the length of a projection to $(1,1)$ keeps constant under the action of this matrix for any $t$. This means that the function $|z+w|^2$ is constant along the trajectories. Compute this function in the coordinate system $(\theta,\varphi,\tau)$:

$$
|z+w|^2 = |\cos(\pi\tau/2)e^{\mathbbm i \varphi}+\sin(\pi\tau/2)e^{-\mathbbm i \theta}|^2 = 1 + \sin(\pi\tau)\cos(\varphi+\theta).
$$

Along $\widehat{v_i}$ the $\theta$- and $\varphi$- coordinates are constant and $\varphi+\theta \in (3\pi/2;5\pi/2)$ since $\theta\in(3\pi/2;2\pi)$ and $\varphi\in(0;\pi/2)$. So we see that in two distinct points of $\widehat{v_i}$ the values of the function $|z+w|^2$ are distinct. Hence the trajectories of two distinct points of $\widehat {v_i}$ do not intersect.

It remains to show that the trajectories are not self-intersecting. Note that the existence of such self-intersection means that the matrix has an eigenvalue 1 for some $t\in(0;\pi/2)$. This is not the case since the eigenvalues are $e^{\pm\mathbbm i t}.$

\noindent{\it Condition (4)}.
Let $(\theta', \varphi', \tau')$ be the image of the point $(\theta, \varphi, \tau)\in \widehat{v_i}$ under $H(\bullet, t).$ Then

\begin{multline*}
\varphi' = \arg z' = \arg(z\cos t + \mathbbm i w\sin t) = \arg\left(\cos t\cos(\pi\tau/2)e^{\mathbbm i \varphi} + \mathbbm i \sin t\sin(\pi\tau/2)e^{-\mathbbm i \theta}\right) = \\ = \arg\left(\cos t \cos(\pi\tau/2)e^{\mathbbm i \varphi} + \sin t\sin(\pi\tau/2)e^{\mathbbm i(-\theta+\pi/2)}\right).
\end{multline*}

Since $\sin t$, $\cos t$, $\sin(\pi\tau/2)$ and $\cos(\pi\tau/2)$ are nonnegative and $\theta\in(3\pi/2;2\pi)$ and $\varphi\in(0;\pi/2)$, it follows that $\varphi'$ belongs to $[\varphi;-\theta+\pi/2]$ if defined, as desired.

A similar calculation for $\theta'$:

$$
\theta' = -\arg w' = -\arg(\mathbbm i z \sin t + w \cos t) = -\arg\left(\sin t\cos(\pi\tau/2) e^{\mathbbm i(\varphi + \pi/2)} + \cos t\sin(\pi\tau/2) e^{-\mathbbm i\theta}\right).
$$

\noindent{\it Condition (5)}.
The boundary of the disk $d_i$ is a union of $\widehat{v_i},$ $\widehat{F(v_i)},$ and two trajectories of endpoints of $\widehat{v_i}$ under the action of $H(\bullet, t),$ where $t\in[0;1].$ Let us prove that the trajectory of the endpoint with $\tau = 1$ coincides with the segment $\widehat{(\theta_i,-\theta_i+\pi/2)}\subset \widehat{\partial r_i}$.

The calculation for the condition~(4) shows that if $\tau = 1$ then the $\theta$- and $\varphi$- coordinates of the point on the trajectory keep constant and equal $\theta_i$ and  $-\theta_i+\pi/2$ respectively for nonzero values of $t$. This means that the trajectory lies inside the segment $\widehat{(\theta_i,-\theta_i+\pi/2)}\subset \widehat{\partial r_i}$. Moreover, the trajectory and the segment coincide because the trajectory is embedded and they have common endpoints. 

The case of $\tau = 0$ and the vertex $(-\varphi_i-\pi/2, \varphi_i)$ of the rectangle $r_i$ is similar.

\end{proof}

\begin{coro}
Let $\widehat r_{\scriptscriptstyle\diagup}$ be a homeomorphism $(\theta,\varphi,\tau)\mapsto (\varphi, \theta, 1-\tau)$.
Then for any rectangular diagram $R$ the morphism $[\widehat r_{\scriptscriptstyle\diagup}]_{\widehat{R}}$ can be decomposed into morphisms associated with exchange moves and (de)stabilizations of type~I.
\end{coro}

\begin{proof}
Define two self-homeomorphisms of $\mathbb T^2$: 
$$
r_{\scriptscriptstyle\diagup}:(\theta,\varphi)\mapsto(\varphi,\theta) \text{ and } r_{\scriptscriptstyle |}:(\theta,\varphi)\mapsto(-\theta,\varphi).
$$

Let $\widehat r_{\scriptscriptstyle |}$ be a homeomorphism $(\theta,\varphi,\tau)\mapsto (-\theta, \varphi, \tau)$.

If we apply Lemma~\ref{symmetry-morphism} to the morphism $[\widehat r_{\scriptscriptstyle\diagdown}]_{\widehat r_{\scriptscriptstyle |}(\widehat{R})}$ we get

$$
[\widehat r_{\scriptscriptstyle\diagdown}]_{\widehat r_{\scriptscriptstyle |}(\widehat{R})} = \widehat s_N\circ\dots\circ\widehat s_2\circ\widehat s_1,
$$

where $R_{i-1} \stackrel{s_i}{\longmapsto} R_i$ is an exchange move or (de)stabilization of type~II, $R_0 = r_{\scriptscriptstyle |}({R})$ and $R_N = r_{\scriptscriptstyle |}\circ r_{\scriptscriptstyle\diagup}({R}).$

Since $\widehat r_{\scriptscriptstyle\diagup} = \widehat r_{\scriptscriptstyle |}\widehat r_{\scriptscriptstyle\diagdown}\widehat r_{\scriptscriptstyle |}$ and $\left(\widehat r_{\scriptscriptstyle |}\right)^2=\mathrm{id}_{\mathbb S^3},$

$$
[\widehat r_{\scriptscriptstyle\diagup}]_{\widehat R} = [\widehat r_{\scriptscriptstyle |}\widehat r_{\scriptscriptstyle\diagdown}\widehat r_{\scriptscriptstyle |}]_{\widehat R} = 
\left[\widehat r_{\scriptscriptstyle |}\widehat s_N\widehat r_{\scriptscriptstyle |}\right]_{\widehat r_{\scriptscriptstyle |}(\widehat R_{N-1})}\circ\dots\circ\left[\widehat r_{\scriptscriptstyle |}\widehat s_2\widehat r_{\scriptscriptstyle |}\right]_{\widehat r_{\scriptscriptstyle |}(\widehat R_{1})}\circ\left[\widehat r_{\scriptscriptstyle |}\widehat s_1\widehat r_{\scriptscriptstyle |}\right]_{\widehat r_{\scriptscriptstyle |}(\widehat R_{0})}.
$$

If $R_{i-1}\mapsto R_i$ is an exchange move ((de)stabilization of type~II) then $r_{\scriptscriptstyle |}(R_{i-1})\mapsto r_{\scriptscriptstyle |}(R_i)$ is an exchange move ((de)stabilization of type~I, respectively). So we obtained a sought-for composition of morphisms.

\end{proof}

%

\begin{rema}
Sequences of elementary moves in Figures~\ref{7_4-flype} and \ref{9_48-flype} are flypes in the sense of~\cite{dyn2003}. So we can say that in each topological type $7_4,$ $9_{48}$ and $10_{136}$ the symmetry group can be generated by a single flype of a minimal diagram.
\end{rema}

The proof of the following lemma is easy and left to the reader.

\begin{lemm}
\label{isotopy-via-d2moves}
Let $L$ be a link and $H:L \times [0;1] \to \mathbb S^3$ be an isotopy of this link such that $H(L\times\{0\})=\mathrm{id}_L$. Suppose that the annuli $L\times[0;1]$ are tiled into disks $\{D_i\}_{i=1,\dots,N}$ and that $H\big|_{D_i}$ is an embedding for each $i=1,\dots,N$. Suppose that
there exists a sequence of links and morphisms
$$L = L_0 \stackrel{\eta_1}{\longmapsto} L_1 \stackrel{\eta_2}{\longmapsto}\dots \stackrel{\eta_N}{\longmapsto} L_N = H(L\times\{1\})$$ such that the triple $(L_{i-1}, L_i, \eta_i)$ is a $\mathbb D^2$-move associated with the disk $H(D_i)$ for any $i=1,\dots, N$. Then the morphism from $L$ to $H(L\times\{1\})$ represented by the isotopy $H$ equals $\eta_N\circ\dots\circ\eta_2\circ\eta_1.$ All links, the isotopy and the tiling are assumed to be PL.
\end{lemm}

\section*{Acknowledgments}

We are grateful to Ivan Dynnikov for support and for providing us the diagram $9_{48}^6$ and a generator of $\mathrm{Sym}(7_4).$

We are grateful to the referee for the careful reading of the paper and for the comments and detailed suggestions.

The work is supported by the Russian Science Foundation under grant 22-11-00299.

\end{document}